\documentclass[a4paper,11pt]{article}

\usepackage{color,url}
\usepackage{enumitem}
\usepackage{graphicx}

\usepackage{psfrag}
\usepackage{epic}
\usepackage{amssymb}
\usepackage{epsfig}
\usepackage{pstricks}
\usepackage{stmaryrd,url}
\usepackage{tikz,amsmath}
\usepackage{diagbox}
\usepackage{framed,multirow}
\usepackage{epstopdf}

\usepackage{amsmath}  
\usepackage{amssymb}   
\usepackage{amsthm} 
\usepackage{comment}
\usepackage{bm}

\usepackage{colortbl}

\newcommand{\stkout}[1]{\ifmmode\text{\sout{\ensuremath{#1}}}\else\sout{#1}\fi}

\numberwithin{equation}{section}

\usepackage[top=1.0in,bottom=1.0in,left=1.0in,right=1.0in]{geometry}
\usepackage[font=scriptsize,labelfont=bf,width=0.85\textwidth]{caption}
\usepackage{subcaption}
\usepackage[colorlinks=true,linkcolor=blue]{hyperref}
\usepackage{makecell}
\usepackage{framed,multirow}
\usepackage{diagbox}

\usepackage{algorithmic,algorithm,xpatch}
\xpatchcmd{\algorithmic}{\setcounter}{\algorithmicfont\setcounter}{}{}
\providecommand{\algorithmicfont}{}

\newlength{\bibitemsep}
\newlength{\bibparskip}
\let\oldthebibliography\thebibliography
\renewcommand\thebibliography[1]{%
  \oldthebibliography{#1}%
  \setlength{\parskip}{\bibitemsep}%
  \setlength{\itemsep}{\bibparskip}%
}

\newtheorem{truth}{Theorem}

\newtheorem{remark}{Remark}
\newtheorem{lemma}{Lemma}

\graphicspath{ {./images/} }

\title{A Fully Discrete Energy-Based Discontinuous Galerkin Method for Variable-Order Time-Fractional Wave Equations
}

\begin{document}
\author{
Lu Zhang 
\thanks{Department of Computational Applied Mathematics and Operations Research, Rice University, Houston, TX 77005, USA. Email: lz82@rice.edu}
 \thanks{Ken Kennedy Institute, Rice University, Houston, TX 77005, USA}}
\maketitle

\begin{abstract}

{Variable-order time-fractional wave equations provide a flexible model for wave phenomena with evolving memory effects and anomalous temporal dynamics. Their numerical approximation is challenging because the variable-order fractional derivative generates time-dependent history weights and therefore lacks the standard time-translation-invariant convolution structure of constant-order fractional operators. In this paper, we develop and analyze a fully discrete energy-based discontinuous Galerkin (DG) method for wave equations with a Caputo-type variable-order time-fractional derivative. The equation is reformulated as a reduced first-order-in-time system, discretized in space by an energy-based DG method, and advanced in time using a second-order approximation of the variable-order Caputo derivative at a specially chosen point in each time interval. The main analytical novelty is a cumulative weight-variation estimate for the variable-order memory weights, which requires only that the variable order $\alpha:[0,T] \rightarrow (0,1)$ be Lipschitz continuous. Based on this estimate, we establish energy stability of the fully discrete scheme and derive second-order temporal convergence together with energy-norm spatial error estimates. The analysis gives suboptimal convergence on general affine simplicial or tensor-product meshes and optimal convergence under additional Cartesian and flux assumptions. Numerical experiments in one and two dimensions validate the theoretical findings.}


\end{abstract}


\textbf{AMS subject }: 35R11, 65M12, 65M15,	65M60

\section{Introduction}

Fractional partial differential equations (PDEs) have been widely used to model processes with memory and hereditary effects, including viscoelasticity, anomalous diffusion, and materials science (see, e.g., \cite{bagley1983theoretical,metzler2000random,pipkin2012lectures,zaslavsky2002chaos}). In many applications, the strength of memory may evolve in time or vary across the medium (see, e.g., \cite{coimbra2003mechanics,zhuang2009numerical}), which motivates variable-order fractional models. Compared with constant-order operators, variable-order fractional derivatives generate time-dependent history kernels and generally lack a standard time-translation-invariant convolution structure. This creates substantial challenges for numerical approximation and analysis, especially for wave-type problems where stability is closely tied to the underlying energy structure.

\subsection{Problem formulation}

This paper focuses on the numerical approximation of variable-order time-fractional wave models, where the fractional order is allowed to vary in time. Specifically, we study a class of wave equations in which \(u=u({\bf x},t)\) depends on location \({\bf x}\) and time \(t\), and the governing equation contains both a standard second-order time derivative and a variable-order Caputo time-fractional derivative:
\begin{equation}\label{eq:problem}
\begin{aligned}
& \frac{\partial^2 u({\bf x}, t)}{\partial t^2}+{}_0^C \mathcal{D}_t^{\beta(t)} u({\bf x}, t) - \Delta u({\bf x}, t)=f({\bf x}, t), \quad({\bf x}, t) \in \Omega \times(0, T],
\end{aligned}
\end{equation}
equipped with the initial conditions
\begin{equation*}
u({\bf x},0)=u_0({\bf x}), 
\qquad 
\frac{\partial u}{\partial t}({\bf x},0)=v_0({\bf x}),
\qquad {\bf x}\in \Omega.
\end{equation*}
Here, \(\Omega\subset\mathbb R^d\), \(d\ge1\), is a bounded domain with a piecewise smooth boundary $\partial\Omega$, \(\Delta\) is the Laplace operator, and \(f({\bf x},t)\) is a given source term. For ease of presentation and analysis, we impose periodic boundary conditions on the computational domain. The operator \({}_0^C\mathcal{D}_t^{\beta(t)}\), with \(\beta(t)\in(1,2)\), denotes the Caputo-type fractional derivative \cite{podlubny1998fractional}, defined for a sufficiently smooth function \(w\) by
\begin{equation*}\label{eq:caputo_2ndorder}
{ }_{0}^{C}\mathcal{D}_t^{\beta(t)} w(t) :=\frac{1}{\Gamma\big(2-\beta(t)\big)} \int_{0}^{t} \frac{w''(s)}{(t-s)^{\beta(t)-1}} ds,
\end{equation*}
where $\Gamma(\cdot)$ is the gamma function and $w''(s)$ is the second-order derivative of $w(s)$ with respect to $s$. When the fractional derivative term is omitted, \eqref{eq:problem} reduces to the classical acoustic wave equation. Fractional time derivatives are often used to model viscoelastic damping and hereditary memory effects, where the material response depends on its deformation history
\cite{nigmatullin1986realization}. In evolving media, the strength of this memory effect may vary with time, leading naturally to variable-order fractional wave models.

A common strategy for treating \eqref{eq:problem} is to reduce it to a first-order-in-time system. Let
\[\alpha(t) := \beta(t) - 1 \in (0,1), \quad v({\bf x}, t) := \frac{\partial u({\bf x}, t)}{\partial t}.\]
Then,
\[{}_0^C \mathcal{D}_t^{\beta(t)} u({\bf x}, t) = {}_0^C \mathcal{D}_t^{\alpha(t)} v({\bf x}, t),\]
where ${}_0^C \mathcal{D}_t^{\alpha(t)} v({\bf x},t)$ is the Caputo-type fractional derivative of order $\alpha(t) \in (0,1)$ defined as
\begin{equation*}\label{eq:caputo_def}
{}_0^C \mathcal{D}_t^{\alpha(t)}w(t) := \frac{1}{\Gamma\big(1-\alpha(t)\big)} \int_0^t \frac{w'(s)}{(t-s)^{\alpha(t)}}\ ds.
\end{equation*}
Thus, \eqref{eq:problem} can be reformulated as the first-order-in-time system
\begin{equation}\label{eq:2nd_system}
\left\{
\begin{aligned}
&\frac{\partial u({\bf x}, t)}{\partial t} = v({\bf x}, t) \\
&\frac{\partial v({\bf x}, t)}{\partial t} = -{}_0^C\mathcal{D}_t^{\alpha(t)} v({\bf x}, t)+\Delta  u({\bf x}, t) + f({\bf x}, t).
\end{aligned}
\right.
\end{equation}
This reformulation is useful both analytically and computationally because the fractional memory term acts on the velocity variable.

\subsubsection{Numerical methods for variable-order fractional PDEs}
 
Over the past few decades, many numerical methods have been developed to solve variable-order fractional PDEs, including finite difference methods \cite{lin2009stability,zhuang2009numerical}, finite element methods \cite{lei2023finite,zheng2021optimal}, spectral methods \cite{zeng2015generalized,zhao2019multi}, and meshless methods \cite{tayebi2017meshless,xu2023novel}, see also \cite{fu2019robust,garrappa2023computational,haq2019numerical,heydari2019wavelet,heydari2019computational,shekari2019meshfree,zeng2017generalized} for related developments. These works demonstrate the broad applicability of numerical methods for variable-order fractional models. However, compared with variable-order diffusion-type equations, the numerical analysis of variable-order fractional wave equations is less developed. 

For variable-order fractional wave models, Bhrawy et al. \cite{bhrawy2016space} proposed a collocation method based on shifted Jacobi--Gauss--Lobatto points in space and shifted
Jacobi--Gauss--Radau points in time for one-dimensional space-time variable-order fractional wave equations. For the reformulated system \eqref{eq:2nd_system}, Zheng and Wang \cite{zheng2022analysis} proposed a numerical method based on the L1 discretization in time and finite element approximation in space, and proved first-order convergence in time and second-order convergence in space. Du et al. \cite{du2022temporal}, motivated by the temporal approximation in \cite{du2020temporal}, developed a second-order finite difference method in time by choosing a special point in each time interval to approximate the Caputo-type variable-order derivative. They further combined this temporal approximation with second-order alternating direction implicit (ADI) and fourth-order compact ADI finite difference schemes in space.

{
These works provide important benchmarks for variable-order fractional wave models. Most existing approaches, however, focus primarily on the temporal approximation of the variable-order fractional derivative, often coupled with classical finite difference, finite element, or spectral spatial discretizations. From the perspective of wave propagation, it is important to develop spatial discretizations compatible with the intrinsic energy structure of the second-order wave operator and to understand how this energy mechanism interacts with the time-dependent memory effects induced by the variable-order Caputo derivative. This viewpoint naturally motivates the use of energy-based discontinuous Galerkin (DG) methods.
}

\subsubsection{Energy-based DG methods}

{For wave equations, a central issue is to design spatial discretizations that are compatible with the underlying energy structure. DG methods are particularly attractive in this context because their elementwise formulation and numerical fluxes provide a natural way to couple local approximations while controlling the discrete energy. In this paper, the DG framework is used primarily as an energy-compatible discretization for wave-dominated problems, rather than merely for geometric flexibility.}

The energy-based DG method introduced in \cite{appelo2015new} was developed for wave equations written directly in second-order form. Compared with local DG methods \cite{chou2014optimal}, it introduces only one auxiliary variable regardless of the spatial dimension. Compared with interior penalty DG methods \cite{riviere2003discontinuous}, it allows the use of mesh-independent conservative or upwind numerical fluxes. These features make the method well suited for energy-stable approximations of second-order wave equations. Energy-based DG methods have been successfully applied to a variety of wave models, including elastic wave equations \cite{appelo2018energy}, advective wave equations \cite{zhang2019energy}, Maxwell's equations \cite{hagstrom2021discontinuous}, semilinear wave equations, and nonlinear
Schr\"odinger equations with wave operator \cite{appelo2020energy,ren2023energy}. Further developments include improved convergence estimates \cite{du2019convergence}, techniques for reducing numerical stiffness caused by polynomial approximations \cite{appelo2021stagger,zhang2021energy}, and extensions to fourth-order semilinear wave equations \cite{zhang2021local}.

{Existing energy-based DG analyses mainly concern classical wave equations without variable-order fractional memory. For the model considered here, the stability mechanism must account not only for the DG spatial fluxes, but also for the time-dependent history weights generated by the variable-order Caputo derivative. This coupling between the variable-order memory term and the DG energy structure is the main analytical issue addressed in this work.}

\subsection{Main results}
{This paper develops and analyzes a fully discrete energy-based DG method for the variable-order time-fractional wave equation \eqref{eq:problem}. The main difficulty comes from the simultaneous presence of a second-order wave operator and a variable-order history term. In contrast to constant-order fractional models, the variable-order Caputo derivative does not lead to a standard time-translation-invariant convolution structure. Consequently, the associated discrete memory weights vary with both the current time level and the history lag, which makes the stability and convergence analysis substantially more delicate.}

{To construct the method, we apply an energy-based DG discretization in space to the reduced system \eqref{eq:2nd_system}. For the variable-order Caputo derivative in time, we use the temporal approximation introduced in \cite{du2020temporal}, in which a special point in each time interval is chosen to obtain a second-order approximation under suitable regularity assumptions. The resulting scheme is spatially high-order and second-order accurate in time.}

{The main analytical contributions of this work are twofold. First, compared with \cite{du2022temporal}, where the stability and convergence analysis relies on the monotonicity condition \(\alpha'(t)\le0\) to compare the discrete memory weights at consecutive time levels, we remove this restriction. Instead, we establish a cumulative weight-variation estimate for the variable-order weights \(a_i^{(m)}\) defined in \eqref{def:a}. This estimate controls the possible increase of the weights after summation over the history levels, rather than requiring a pointwise comparison at every time step. As a result, the stability and convergence analysis can be carried out under the weaker assumption
\begin{equation}\label{eq:condition_alpha}
\alpha:[0,T]\to(0,1) \text{ is Lipschitz continuous}.
\end{equation}
Second, unlike existing energy-based DG analyses for classical wave equations, the present analysis must also control the interaction between the time-dependent memory term and the DG spatial discretization, including numerical traces, jumps, flux contributions, and projection errors. } 

Under the assumption \eqref{eq:condition_alpha}, we prove energy stability of the proposed fully discrete scheme for suitable mesh-independent numerical fluxes. We then establish second-order convergence in time and derive energy-norm error estimates for the spatial discretization. For general affine simplicial and tensor-product meshes, the scheme yields a suboptimal energy-norm convergence estimate under the condition \(q_u-2\le q_v\le q_u\), where \(q_u\) and \(q_v\) denote the polynomial degrees used to approximate the displacement \(u\) and the velocity \(v\), respectively. In a Cartesian setting, and for specific choices of numerical fluxes, the estimate can be improved to an optimal energy-norm convergence result when \(q_v=q_u-1\); see Remark \ref{remark:optimal}.

The rest of the paper is organized as follows. Section \ref{sec:full_discrete} introduces the fully discrete energy-based DG scheme for \eqref{eq:problem}. Section \ref{sec:stability} proves energy stability of the scheme and derives the corresponding energy-norm error estimates. Numerical experiments in
Section \ref{sec:numerical} confirm the predicted convergence behavior in both 1D and 2D settings. Finally, Section \ref{sec:discussion} concludes the paper.

\section{Fully Discrete Energy-based DG Scheme}\label{sec:full_discrete}

In this section, we introduce the fully discrete energy-based DG method for the reformulated problem \eqref{eq:2nd_system}.

\subsection{Spatial semi-discrete energy-based DG formulation}
{Let $\mathcal T_h$ be a tessellation of the computational domain $\Omega_h$, consisting of non-overlapping elements $K$ such that $\bar{\Omega}_h=\cup_{K\in\mathcal T_h}\bar K$. Throughout the theoretical analysis, we assume that \(\mathcal{T}_h\) is a shape-regular and quasi-uniform family of affine simplicial or tensor-product meshes. These assumptions ensure that the Bramble--Hilbert approximation estimates, as well as the inverse and trace inequalities used in Section \ref{sec:error_full} hold with constants independent of \(h:=\max_{K\in\mathcal T_h} h_K\).} 

{To approximate the solution components $(u,v)$, we use piecewise polynomial spaces of degrees $q_u$ and $q_v$, respectively, where 
$(q_u,q_v)\in \mathcal{Q}_{\text{adm}}$ defined in \eqref{eq:degree_admissible}. On each element $K$, we denote by $\mathcal P^s(K)$ either the space of polynomials of total degree at most $s$ on a simplex, or the tensor-product polynomial space of degree at most $s$ in each coordinate on a tensor-product element. We then define
\begin{equation*}
\begin{aligned}
{U_h^{q_u}} &{:= \{w({\bf x})\big| w({\bf x}) \in \mathcal{P}^{q_u}(K), {\bf x} \in K\quad \forall K \in \mathcal{T}_h \},}\\
{V_h^{q_v}} & {:= \{w({\bf x})\big| w({\bf x}) \in \mathcal{P}^{q_v}(K), {\bf x} \in K\quad \forall K \in \mathcal{T}_h \}.}
\end{aligned}
\end{equation*}
We note that the stability analysis in Section \ref{sec:h1stability} holds for arbitrary \((q_u,q_v)\in \mathbb{N}\times\mathbb{N}_0\); while for the error estimate in Section \ref{sec:error_full}, we impose the additional degree-compatibility condition stated in \eqref{eq:degree_admissible}.}

To derive an energy-based DG scheme for the system (\ref{eq:2nd_system}), for each element $K\in\mathcal{T}_h$ we test the first and second equations in (\ref{eq:2nd_system}) by functions $-\Delta \phi, \psi$ with $\phi\in U_h^{q_u}$ and $\psi\in V_h^{q_v}$, respectively, and add boundary integrals vanishing for continuous problem to get
\begin{equation*}
\int_K -\Delta \phi \Big(\frac{\partial u_h}{\partial t} - v_h\Big) = \int_{\partial K} \nabla \phi \cdot{\bf n} \Big(v_h^\ast - \frac{\partial u_h}{\partial t}\Big)\ dS,
\end{equation*}
and
\begin{equation*}
\int_K \psi \frac{\partial v_h}{\partial t} -\psi \Delta u_h + \psi {}_0^C\mathcal{D}_t^{\alpha(t)} v_h\ d {\bf x}= \int_K  \psi f({\bf x}, t) \ d{\bf x} + \int_{\partial K} \psi \big((\nabla u_h)^\ast\cdot {\bf n} - \nabla u_h \cdot {\bf n}\big)\ dS,
\end{equation*}
where $v_h^\ast$ and $(\nabla u_h)^\ast$ are numerical fluxes to be specified on the element boundaries, and $\bf n$ is the outward unit normal to $\partial K$.  An integration by parts leads to 
\begin{equation}\label{DG_scheme1}
\begin{aligned}
\int_K \nabla \phi \cdot \nabla \Big(\frac{\partial u_h}{\partial t} - v_h\Big) \ d{\bf x} &= \int_{\partial K} \nabla \phi \cdot{\bf n} (v_h^\ast - v_h)\ dS,\\
\int_K \psi \frac{\partial v_h}{\partial t} + \nabla\psi \cdot\nabla u_h +\psi {}_0^C\mathcal{D}_t^{\alpha(t)} v_h \ d {\bf x} &= \int_K \psi f({\bf x}, t) \ d{\bf x} + \int_{\partial K} \psi (\nabla u_h)^\ast\cdot {\bf n} \ dS.
\end{aligned}
\end{equation}
Moreover, to determine the mean value of $\frac{\partial u_h}{\partial t}$, we impose
\begin{equation*}
\int_K \tilde{\phi} \Big(\frac{\partial u_h}{\partial t} - v_h\Big) d{\bf x} = 0\quad \forall \tilde\phi \in U_h^0.
\end{equation*}

Next, to complete the energy-based DG formulation, we need to specify the numerical fluxes $v_h^\ast$ and $(\nabla u_h)^\ast$ at the element boundaries. Specifically, we label two elements sharing an interelement boundary face, denoted as $F$, by $K^+$ and $K^-$. Furthermore, let $\chi/\bm \chi$ be a continuously differentiable scalar/vector value function on $K^+$ and $K^-$, and $\chi^{\pm} := (\chi|_{K^{\pm}})|_{F}$, ${\bm \chi}^{\pm} := ({\bm \chi}|_{K^{\pm}})|_{F}$. We then introduce standard notations for jumps given by 
\[[\chi]:= \chi^-{\bf n}^- + \chi^+{\bf n}^+, \quad [\bm\chi]:= \bm\chi^-\cdot {\bf n}^- +  \bm\chi^+\cdot{\bf n}^+.\]
Here, ${\bf n}^+$ and ${\bf n}^-$ are the unit outward normals to $\partial K^+$ and $\partial K^-$, respectively. We now define the numerical fluxes as follows with mesh-independent constants $\theta \in \mathbb{R}$, $\gamma, \zeta \geq 0$ (see, e.g., \cite{appelo2015new})
\begin{equation}\label{flux1}
v_h^\ast := \theta v_h^+ + (1-\theta) v_h^- - \zeta[\nabla u_h], \quad (\nabla u_h)^\ast := (1-\theta) \nabla u_h^+ + \theta\nabla u_h^- - \gamma [v_h].
\end{equation} 

To obtain a fully discrete energy-based DG scheme for system \eqref{eq:2nd_system}, we approximate both the integer-order and variable-order fractional time derivatives in \eqref{DG_scheme1}. In the following, we present the time discretization scheme employed in this work.

\subsection{The time discretization}

We discretize the time interval $[0,T]$ using grid points $t_m := m\tau$ for $m=0,1,2,\cdots,M$, where $\tau = T/M$ is the uniform time step size. In the following, we introduce some important approximations and lemmas relevant to the temporal discretization, stability, and error estimation of the proposed scheme.
\begin{lemma}\label{lemma:nonlinear}
(Lemma 2.1 of \cite{du2020temporal}) Suppose that $\alpha(t)$ satisfies the condition \eqref{eq:condition_alpha}, and let $m = 0,1,\cdots,M-1$. Define
\begin{equation}\label{eq:F}
F(\sigma) := \sigma - \Big(1 - \frac{1}{2}\alpha(t_m + \sigma\tau)\Big).
\end{equation}
Then, for sufficiently small $\tau > 0$ such that $L_\alpha \tau < 2$, the equation $F(\sigma)=0$ has a unique root
\[\sigma_m := \sigma_m(\tau) \in \Big(\frac{1}{2}, 1\Big).\]
Here, $L_\alpha$ denotes the Lipschiz constant of $\alpha(t)$.
\end{lemma}

For each $m$, further define
\[t_{m+\sigma_m} := t_m + \sigma_m \tau, \quad m = 0,1,\cdots,M-1.\]
We then have the following lemma for the approximation of variable-order fractional derivatives.
\begin{lemma}\label{lemma:fractional_approximation}
(\cite{du2020temporal}) Assume that $\ell(t) \in W^{3,\infty}(0,T)$ and $\alpha(t)$ satisfies the condition \eqref{eq:condition_alpha}. Define
\begin{equation*}\label{eq:sm}
\alpha_{m+\sigma_m}:=\alpha\left(t_{m+\sigma_m}\right),\ \ell^m:=\ell\left(t_m\right),\ s_m:= \begin{cases}2^{1-\alpha_{1 / 2}} \tau^{\alpha_{1 / 2}} \Gamma\left(2-\alpha_{1 / 2}\right), m=0, \\ \tau^{\alpha_{m+\sigma_m}} \Gamma\left(2-\alpha_{m+\sigma_m}\right), 1 \leq m \leq M-1,\end{cases}
\end{equation*}
where $\sigma_m$ is the root of the nonlinear problem \eqref{eq:F}. Then, when $m = 0$, there exists a positive constant $0<C_1<\infty$ independent of $\tau$ such that
\begin{equation*}\label{eq:Cap_1}
\bigg|{}_0^C \mathcal{D}_t^{\alpha(t)} \ell(t) \big|_{t = t_{1/2}} -  \frac{\ell^1 - \ell^0}{s_0}\bigg| \leq C_1\|\ell''(t)\|_{L^\infty(0,t_1)}\mathcal{O}(\tau^{2-\alpha_{1/2}}).
\end{equation*}
Furthermore, when $1\leq m\leq M$ and $0\leq i\leq m$, define
\begin{equation}\label{def:a}
    a_i^{(m)}:=\frac{c_i^{(m)}}{s_m},
\end{equation}
we have, for any $1\leq m\leq M-1$,
\begin{equation*}\label{eq:Cap_2}
\bigg|{ }_0^C \mathcal{D}_t^{\alpha(t)} \ell(t)\big|_{t=t_{m+\sigma_m}} -  \sum_{i=0}^m a_{m-i}^{(m)}\left(\ell^{i+1}-\ell^i\right)\bigg| \leq C_2 \|\ell'''(t)\|_{L^\infty(0,t_{m+1})}\tau^{3-\alpha_{m+\sigma_m}},
\end{equation*}
where $C_2 \in (0,\infty)$ is a positive constant independent of $\tau$, and 
\begin{equation}
c_i^{(m)} \!\!=\! \begin{cases}
\begin{aligned}
&\frac{\left(1+\sigma_m\right)^{2-\alpha_{m+\sigma_m}}-\sigma_m^{2-\alpha_{m+\sigma_m}}}{2-\alpha_{m+\sigma_m}} \\
&\quad -\frac{1}{2}\left[\left(1+\sigma_m\right)^{1-\alpha_{m+\sigma_m}}-\sigma_m^{1-\alpha_{m+\sigma_m}}\right], 
\end{aligned} & i = 0, \\[3ex]
\begin{aligned}
&\frac{\left(i+\sigma_m+1\right)^{2-\alpha_{m+\sigma_m}}-2\left(i+\sigma_m\right)^{2-\alpha_{m+\sigma_m}}+\left(i+\sigma_m-1\right)^{2-\alpha_{m+\sigma_m}}}{2-\alpha_{m+\sigma_m}} \\
&\  -\frac{\left(i+\sigma_m+1\right)^{1-\alpha_{m+\sigma_m}}-2\left(i+\sigma_m\right)^{1-\alpha_{m+\sigma_m}}+\left(i+\sigma_m-1\right)^{1-\alpha_{m+\sigma_m}}}{2}, 
\end{aligned} & 1\!\leq i \leq \!m\!-\!1, \\[3ex]
\begin{aligned}
&\frac{3\left(i+\sigma_m\right)^{1-\alpha_{m+\sigma_m}}-\left(i+\sigma_m-1\right)^{1-\alpha_{m+\sigma_m}}}{2} \\
&\quad -\frac{\left(i+\sigma_m\right)^{2-\alpha_{m+\sigma_m}}  -\left(i+\sigma_m-1\right)^{2-\alpha_{m+\sigma_m}}}{2\alpha_{m+\sigma_m}},
\end{aligned} & i=m.
\end{cases}\label{eq:cdef}
\end{equation}
\end{lemma}

Moreover, we have the following lemma for the boundness of the coefficients in the above approximation of Caputo-type variable-order fractional derivatives.
\begin{lemma}\label{lemma:ineq}
(Lemma 4 of \cite{alikhanov2015new} and Lemma 4.3 of \cite{du2022temporal}) Suppose that $\alpha(t)$ satisfies the condition \eqref{eq:condition_alpha}. Then, the coefficients $\{c_i^{(m)}\}_{i=0}^m$ given in \eqref{eq:cdef} satisfy the following relation:
\begin{equation*}\label{eq:sequence_c}
 0 <\frac{1-\alpha_{m+\sigma_m}}{2\left(m+\sigma_m\right)^{\alpha_{m+\sigma_m}}} < c_m^{(m)} < \cdots < c_1^{(m)} < c_0^{(m)} \qquad \forall 1\leq m\leq M-1. 
\end{equation*}
Furthermore, there exist positive constants $0 < C_3,C_4 < \infty$ independent of $\tau$ such that
\begin{equation*}\label{eq:c1c2}
\tau \sum_{i=1}^m a_{i-1}^{(i)} - a_{i}^{(i)} \leq C_3<\infty \quad \text{and} \quad \tau \sum_{i=1}^m a_i^{(i)}  \leq C_4<\infty \qquad \forall 1\leq m\leq M-1,
\end{equation*}
where $a_{i-1}^{(i)}$ and $a_i^{(i)}$ are defined in \eqref{def:a}.
\end{lemma}

{In contrast to Lemma 4.4 of \cite{du2022temporal}, which relies on a pointwise-in-time weight-variation estimate for the weights defined in \eqref{def:a}, we prove below a cumulative weight-variation estimate. This new estimate is sufficient for the stability and error analysis and enables us to remove the restrictive monotonicity condition $\alpha'(t)\leq 0$ required in \cite{du2022temporal}.}
\begin{lemma}\label{lem:cumulative_weight_variation}
{Assume that $\alpha(t)$ satisfies the condition \eqref{eq:condition_alpha}.} Then, for sufficiently small $\tau$, there exists a positive constant $0 < C_5 < \infty$ independent of $\tau$ and the time level $m$, such that for any $m$ with $2 \leq m \leq M-1$,
\begin{equation}\label{eq: aux1}
    \big|a_i^{(m)}-a_i^{(m-1)}\big|
    \le
    C_5\tau
    \Big(1+\big|\log\big((i+1)\tau\big)\big|\Big)
    \big((i+1)\tau\big)^{-\alpha_{\max}} \qquad \forall 0\leq i\leq m-1,
\end{equation}
where $a_i^{(m)}$ is defined in \eqref{def:a} and \[\alpha_{\max} := \max_{t\in[0,T]} \alpha(t) \in (0,1).\] Consequently, for any $2\le k\le n\le M-1$,
\begin{equation}\label{eq: aux2}
    \tau\sum_{m=k}^{n}
    \Big(a_{m-k}^{(m)}-a_{m-k}^{(m-1)}\Big)_+
    \le C_5\tau,
\end{equation}
where $(x)_+ = \max\{x, 0\}$.
\end{lemma}
{\begin{proof}
For readibility, we defer the proof to the Appendix.
\end{proof}}

Last, for the approximation of the first-order and the zero-order time derivatives, we have the following second-order approximation:
\begin{lemma}\label{lemma:integer_approximation}
Assume that {$\ell(t) \in W^{3,\infty}(0,T)$}.  For each $1\leq m\leq M-1$, let us denote
\begin{equation*}\label{eq:frac}
\ell^{m+\sigma_m} := \sigma_m \ell^{m+1} + (1-\sigma_m) \ell^m, \quad \delta_t \ell^{m+\sigma_m} := \frac{(2\sigma_m+1)\ell^{m+1} - 4\sigma_m \ell^m + (2\sigma_m - 1)\ell^{m-1}}{2\tau},
\end{equation*}
where $\sigma_m$ is the unique root of the nonlinear problem \eqref{eq:F}. {Then, there exist positive constants $\{C_i\}_{i=6}^9$ with $C_i\in(0,\infty)$, independent of $\tau$, such that 
\begin{equation*}
\begin{aligned}
& | \ell\left(t_{1 / 2}\right) - \tfrac{1}{2}(\ell^0 + \ell^1) | \leq C_6\tau^2 \|\ell''\|_{L^\infty(0,t_1)}, \quad\quad\ |\ell_t \left(t_{1 / 2}\right) - \tfrac{\ell^1 - \ell^0}{\tau}| \leq C_7\tau^2\|\ell'''\|_{L^\infty(0,t_1)}, \\   
&|\ell\left(t_{m+\sigma_m}\right) - \ell^{m+\sigma_m}| \leq C_8 \tau^2\|\ell''\|_{L^\infty(t_m,t_{m+1})}, \quad \ |\ell_t\left(t_{m+\sigma_m}\right) - \delta_t \ell^{m+\sigma_m}| \leq C_9\tau^2\|\ell'''\|_{L^\infty(t_{m-1},t_{m+1})}. 
\end{aligned}
\end{equation*}}
\end{lemma}

 We are now ready to present the fully discrete energy-based DG scheme for \eqref{eq:2nd_system}.

\subsection{Fully discrete energy-based DG discretization}\label{sec:scheme}
Let $u_h^{m} \in U_h^{q_u}$ and $v_h^{m} \in V_h^{q_v}$ be the approximations of $u_h({\bf x}, t_{m})$ and $v_h({\bf x}, t_{m})$, respectively. Building on the spatial semi-discrete energy-based DG scheme \eqref{DG_scheme1} and the difference approximations for the integer and fractional time derivatives introduced in Lemma \ref{lemma:fractional_approximation} and Lemma \ref{lemma:integer_approximation}, we derive the fully discrete energy-based DG scheme: for all $\phi\in U_h^{q_u}$ and $\psi\in V_h^{q_v}$,

\textit{\textbf{Case I: $m=0$}}
\begin{equation}\label{eq:DG_fully1}
\begin{aligned}
\int_K \nabla \phi \cdot \nabla \Big(\frac{u_h^1 - u_h^0}{\tau} - v_h^{1/2}\Big) \ d{\bf x} 
&= \int_{\partial K} \nabla \phi \cdot{\bf n} \Big((v_h^{1/2})^\ast -v_h^{1/2}\Big)\ dS,\\
\int_K \psi \frac{v_h^1 - v_h^0}{\tau} + \nabla \psi \cdot \nabla u_h^{1/2} + \psi \frac{v_h^1 - v_h^0}{s_0} \ d{\bf x} 
&= \int_K \psi f({\cdot, t_{1/2}})\ d{\bf x} + \int_{\partial K} \psi (\nabla u_h^{1/2})^\ast \cdot {\bf n} dS,
\end{aligned}
\end{equation}

\textit{\textbf{Case II: {$1 \leq m \leq M-1$}}}
\begin{equation}
\begin{aligned}
&\int_K \nabla \phi \cdot \nabla \left(\delta_t u_h^{m+\sigma_m} - v_h^{m+\sigma_m}\right)\ d{\bf x} 
= \int_{\partial K} \nabla \phi \cdot {\bf n} \big( (v_h^{m+\sigma_m})^\ast - v_h^{m+\sigma_m} \big) dS,\\
&\int_K \psi \delta_t v_h^{m+\sigma_m} + \nabla\psi \cdot \nabla u_h^{m+\sigma_m} + \psi \sum_{i=0}^m {a_{m-i}^{(m)}} \big(v_h^{i+1} - v_h^i\big) \ d {\bf x} \\
=& \int_K \psi f(\cdot, t_{m+\sigma_m}) \ d{\bf x} + \int_{\partial K} \psi (\nabla u_h^{m+\sigma_m})^\ast\cdot {\bf n} \ dS. \label{eq:DG_fully1_m}
\end{aligned}
\end{equation}

Moreover, to ensure the solvability of the fully discrete scheme \eqref{eq:DG_fully1}--\eqref{eq:DG_fully1_m} similar to the semi-discrete case, it is necessary to impose additional equations to determine the mean value of the approximation of the first-order time derivative. Specifically, we impose the following discretizations for computational purposes:
\[
\int_K \tilde{\phi} \Big(\frac{u_h^1 - u_h^0}{\tau} - v_h^{1/2}\Big) d{\bf x} = 0\quad \text{and}\quad 
\int_K \tilde{\phi} \left(\delta_t u_h^{m+\sigma_m} - v_h^{m+\sigma_m}\right) d{\bf x} = 0 \quad \forall \tilde{\phi}\in U_h^0,
\]
where $m=1,2,\cdots,M-1$.

In the following sections, we study the stability and error estimates of the fully discrete energy-based DG scheme \eqref{eq:DG_fully1}--\eqref{eq:DG_fully1_m}. We use standard notation for Sobolev spaces and norms. Since the numerical solution is discontinuous across element interfaces, we denote by $\nabla_h$ the elementwise gradient, namely, $(\nabla_h w)|_K := \nabla (w|_K)$ for all $K\in\mathcal T_h$. Thus, $\|\nabla_h w\|_{L^2(\Omega_h)}^2
    :=
    \sum_{K\in\mathcal T_h}\|\nabla w\|_{L^2(K)}^2$. For \(L^2\) norms, we write \(\|w\|_{L^2(\Omega_h)}\), with the understanding that $\|w\|_{L^2(\Omega_h)}^2
    = \sum_{K\in\mathcal T_h}\|w\|_{L^2(K)}^2$ for piecewise \(L^2\) functions. Finally, we use \(C\) to denote a generic positive constant that may vary from line to line but is independent of the time step size \(\tau\), the mesh size \(h\), and the time level $m$.


\section{Stability and Error Estimates}\label{sec:stability}

In this section, we establish the stability and error estimates of the fully discrete energy-based DG scheme \eqref{eq:DG_fully1}--\eqref{eq:DG_fully1_m}.

\subsection{Energy stability}\label{sec:h1stability}

The first set of main results of this work establishes energy stability and is given as follows.

\begin{truth}\label{th:discrete_stability}
{Assume that $\alpha(t)$ satisfies the condition \eqref{eq:condition_alpha}}. Then, for sufficient small $\tau$, there exists a positive constant $C$ independent of the time step size $\tau$ and the spatial mesh size $h$ such that the following inequality holds for $(u_{h}^{m},v_h^m) \in U_{h}^{q_u}\times V_h^{q_v}$, given by the fully discrete energy-based DG scheme \eqref{eq:DG_fully1}--\eqref{eq:DG_fully1_m},
\begin{equation}\label{eq:discrete_stable}
\begin{aligned}
&\|\nabla_h u_h^m\|_{L^2(\Omega_h)}^2 + \left\| v_{h}^{m}\right\|_{L^2(\Omega_h)}^2 \\
\leq &\ C\Big(\| \nabla_h u_{h}^{0}\|_{L^2(\Omega_h)}^2 + \left\| v_{h}^{0}\right\|_{L^2(\Omega_h)}^2 + \tau \|f(\cdot, t_{1/2})\|_{L^2(\Omega_h)}^2+\tau \sum_{i=1}^{m-1}\left\|f(\cdot, t_{i+\sigma_i})\right\|_{L^2(\Omega_h)}^{2}\Big) 
\end{aligned}\quad \forall m,
\end{equation}
where $t_{1/2} = t_0 + \tau/2$ and $t_{i+\sigma_i} = t_i + \sigma_i \tau$ with $\sigma_i$ being the root of \eqref{eq:F}.
\end{truth}

\begin{proof}
We note that the fully discrete scheme \eqref{eq:DG_fully1}--\eqref{eq:DG_fully1_m} adopts different time integration methods for $m = 0$ (see \eqref{eq:DG_fully1}) and ${1\leq m \leq M-1}$ (see \eqref{eq:DG_fully1_m}). We proceed with an analysis for the case of ${1\leq m \leq M-1}$ first (\underline{\textit{step one}}), and then for the case of $m=0$ (\underline{\textit{step two}}).

\textbf{Step one:} $m \geq 1$.

We choose the test functions $\phi = u_h^{m+\sigma_m}$ and $\psi = v_h^{m+\sigma_m}$ in the scheme \eqref{eq:DG_fully1_m}. Summing the resulting equations over all elements $K$ we obtain
\begin{equation}\label{eq:stablity1}
\begin{aligned}
&\sum_{K\in \mathcal{T}_h} \int_{K} \sum_{i=0}^{m} a_{m-i}^{(m)} v_h^{m+\sigma_m}(v_h^{i+1}-v_h^i) d {\bf x}+ \nabla u_h^{m+\sigma_m}\cdot \nabla \delta_tu_h^{m+\sigma_m} + v_h^{m+\sigma_m}\delta_t v_h^{m+\sigma_m}\ d{\bf x} \\
= &\sum_{K\in \mathcal{T}_h} \int_{\partial K} \nabla u_h^{m+\sigma_m} \cdot {\bf n} \big( (v_h^{m+\sigma_m})^\ast - v_h^{m+\sigma_m} \big) + v_h^{m+\sigma_m} (\nabla u_h^{m+\sigma_m})^\ast\cdot {\bf n}\ dS \\
+ &\sum_{K\in \mathcal{T}_h} \int_K v_h^{m+\sigma_m} f(\cdot, t_{m+\sigma_m})\ d{\bf x}.
\end{aligned}
\end{equation}
Let \(\mathcal F_h\) denote the set of all interelement faces, with periodic boundary faces identified according to the periodic boundary condition. Using the numerical fluxes defined in \eqref{flux1} and periodic boundary conditions, we have
\begin{equation}\label{eq:aux1}
\begin{aligned}
&\sum_{K\in \mathcal{T}_h} \int_{\partial K} \nabla u_h^{m+\sigma_m} \cdot {\bf n} \big( (v_h^{m+\sigma_m})^\ast - v_h^{m+\sigma_m} \big) + v_h^{m+\sigma_m} (\nabla u_h^{m+\sigma_m})^\ast\cdot {\bf n}\ dS 
\\
=& \sum_{F\in \mathcal{F}_h}\int_F - \gamma \big|[v_h^{m+\sigma_m}]\big|^2 - \zeta |[\nabla u_h^{m+\sigma_m}]|^2\ dS.
\end{aligned}
\end{equation}
Substituting \eqref{eq:aux1} into (\ref{eq:stablity1}) and using $\gamma,\zeta\geq 0$ yields
\begin{equation}\label{eq:aux2}
\begin{aligned}
& \sum_{K\in\mathcal{T}_h}\int_K \sum_{i=0}^{m} a_{m-i}^{(m)}  v_h^{m+\sigma_m}(v_h^{i+1}-v_h^i) + \nabla u_h^{m+\sigma_m}\cdot  \delta_t \nabla u_h^{m+\sigma_m} + v_h^{m+\sigma_m}\delta_t v_h^{m+\sigma_m} d{\bf x}\\
\leq & \sum_{K\in\mathcal{T}_h}\int_{K} v_h^{m+\sigma_m} f(\cdot, t_{m+\sigma_m})\ d{\bf x}. 
\end{aligned}
\end{equation}
To proceed, we first estimate the second and third terms on the left-hand-side of \eqref{eq:aux2}. {To this end,} we introduce the auxiliary function \(\mathcal{A}(\cdot)\) and provide its corresponding estimates in the following
\begin{equation}\label{eq:A}
\begin{aligned}
\mathcal{A}(w^m) := &\left(2 \sigma_{m-1}+1\right)\left\|w^m\right\|_{L^2(\Omega_h)}^2-\left(2 \sigma_{m-1}-1\right)\left\|w^{m-1}\right\|_{L^2(\Omega_h)}^2\\
&+\left(2 \sigma_{m-1}^2+\sigma_{m-1}-1\right)\left\|w^m-w^{m-1}\right\|_{L^2(\Omega_h)}^2 \\
\geq & \frac{1}{\sigma_{m-1}} \|w_h^m\|_{L^2(\Omega_h)}^{{2}}.
\end{aligned}
\end{equation}
By direct calculation (see, e.g., Lemma 3.5 in \cite{sun2016some}), we find that
\begin{equation}\label{eq:aux3}
\begin{aligned}
\sum_{K\in\mathcal{T}_h}\int_{K} v_h^{m+\sigma_m} \delta_t v_h^{m+\sigma_m} d{\bf x} 
&\geq \frac{1}{4\tau}\Big(\mathcal{A}\big(v_h^{m+1}\big) - \mathcal{A}\big(v_h^{m}\big)\Big),\\
\sum_{K\in\mathcal{T}_h}\int_{K} \nabla u_h^{m+\sigma_m} \cdot \delta_t \nabla u_h^{m+\sigma_m} d{\bf x} 
&\geq \frac{1}{4\tau}\Big(\mathcal{A}\big(\nabla_h u_h^{m+1}\big) - \mathcal{A}\big(\nabla_h u_h^{m}\big)\Big).
\end{aligned}
\end{equation}
Substituting \eqref{eq:aux3} into \eqref{eq:aux2}, we arrive at
\begin{equation}\label{eq:aux4}
\begin{aligned}
&\sum_{K\in\mathcal{T}_h}\int_K \sum_{i=0}^{m} {a_{m-i}^{(m)}}  v_h^{m+\sigma_m}(v_h^{i+1}-v_h^i) \ d {\bf x} + \frac{1}{4\tau}\left(\mathcal{A}\big(v_h^{m+1}\big) + \mathcal{A}\big(\nabla_h u_h^{m+1}\big)\right)\\
\leq & \ \frac{1}{4\tau}\left(\mathcal{A}\big(v_h^{m}\big) + \mathcal{A}\big(\nabla_h u_h^{m}\big)\right) + \sum_{K\in\mathcal{T}_h}\int_K v_h^{m+\sigma_m} f(\cdot, t_{m+\sigma_m})\ d{\bf x}.
\end{aligned}
\end{equation}

To estimate the first term on the left-hand-side of \eqref{eq:aux2}, we have from straightforward calculations
\begin{equation}\label{eq:aux5}
\begin{aligned}
&\frac{1}{2}\Big(a_0^{(m)}\left\|v_h^{m+1}\right\|_{L^2(\Omega_h)}^2-\sum_{i=1}^m\big({a_{m-i}^{(m)}}-{a_{m-i+1}^{(m)}}\big)\left\|v_h^i\right\|_{L^2(\Omega_h)}^2-{a_m^{(m)}}\left\|v^0_h\right\|_{L^2(\Omega_h)}^2\Big) \\
\leq & \sum_{K\in\mathcal{T}_h}\int_K\sum_{i=0}^{m} {a_{m-i}^{(m)}} v_h^{m+\sigma_m}(v_h^{i+1}-v_h^i) \ d {\bf x}.
\end{aligned}
\end{equation}
Plugging \eqref{eq:aux4} and \eqref{eq:aux5} into \eqref{eq:aux2} gives
\begin{equation}\label{eq:aux6}
\begin{aligned}
&\mathcal{A}\big(v_h^{m+1}\big) + \mathcal{A}\big(\nabla_h u_h^{m+1}\big) + 2\tau \sum_{i=2}^{m+1} {a_{m-i+1}^{(m)}} \|v_h^i\|_{L^2(\Omega_h)}^2\\
\leq &\  \mathcal{A}\big(v_h^{m}\big) + \mathcal{A}\big(\nabla_h u_h^{m}\big) + 2\tau\sum_{i=2}^m {a_{m-i}^{(m)}} \|v_h^i\|_{L^2(\Omega_h)}^2 +  4\tau\sum_{K\in\mathcal{T}_h}\int_K v_h^{m+\sigma_m} f(\cdot, t_{m+\sigma_m})\ d{\bf x} \\
&+ 2\tau\Big(\big({a_{m-1}^{(m)}} - {a_m^{(m)}}\big)\left\|v_h^{1}\right\|_{L^2(\Omega_h)}^2 +{a_m^{(m)}}\left\|v^0_h\right\|_{L^2(\Omega_h)}^2\Big).
\end{aligned}
\end{equation}
For $1\leq m\leq M-1$, denote further by
\begin{equation}\label{eq:Q}
\mathcal{Q}^{m+1} := \mathcal{A}\big(v_h^{m+1}\big) + \mathcal{A}\big(\nabla_h u_h^{m+1}\big) + 2\tau \sum_{i=2}^{m+1} {a_{m-i+1}^{(m)}} \|v_h^i\|_{L^2(\Omega_h)}^2,
\end{equation}
{we then obtain from \eqref{eq:aux6} that
\begin{equation}\label{eq: auxx1}
\begin{aligned}
\mathcal{Q}^{m+1} \leq &\ \mathcal{Q}^m + 2\tau \sum_{i=2}^m \Big(a_{m-i}^{(m)} - a_{m-i}^{(m-1)}\Big)_+\|v_h^i\|_{L^2(\Omega_h)}^2+4\tau\sum_{K\in\mathcal{T}_h}\int_K v_h^{m+\sigma_m} f(\cdot, t_{m+\sigma_m})\ d{\bf x} \\
&+ 2\tau\Big(\big(a_{m-1}^{(m)} - a_m^{(m)}\big)\left\|v_h^{1}\right\|_{L^2(\Omega_h)}^2 +a_m^{(m)}\left\|v^0_h\right\|_{L^2(\Omega_h)}^2\Big).
\end{aligned}
\end{equation}
Summing the above inequality in the index $m$ from $2$ to $n$ with $2\leq n\leq M-1$, and using the fact
\[\sum_{m=2}^n\sum_{i=2}^m \Big(a_{m-i}^{(m)} - a_{m-i}^{(m-1)}\Big)_+\|v_h^i\|_{L^2(\Omega_h)}^2 = \sum_{i=2}^n\|v_h^i\|_{L^2(\Omega_h)}^2\sum_{m=i}^n \Big(a_{m-i}^{(m)} - a_{m-i}^{(m-1)}\Big)_+,\]
we then arrive at the following estimate after invoking Lemma \ref{lem:cumulative_weight_variation}
\begin{equation}\label{eq: auxx2}
\begin{aligned}
\mathcal{Q}^{n+1}& \leq \ \mathcal{Q}^2 + 2C_5\tau \sum_{i=2}^n\|v_h^i\|_{L^2(\Omega_h)}^2 + 4\tau \sum_{m=2}^n\sum_{K\in\mathcal{T}_h}\int_K v_h^{m+\sigma_m} f(\cdot, t_{m+\sigma_m})\ d{\bf x}  \\
& + 2\tau\sum_{m=2}^n\bigg(\Big(a_{m-1}^{(m)} - a_m^{(m)}\Big)\left\|v_h^{1}\right\|_{L^2(\Omega_h)}^2 +a_m^{(m)}\left\|v^0_h\right\|_{L^2(\Omega_h)}^2\bigg)\qquad \forall 2\leq n\leq M-1.
\end{aligned}
\end{equation}
On the other hand, from \eqref{eq: auxx1}, we have
\begin{equation}\label{eq: auxx3}
\mathcal{Q}^2 \leq \mathcal{Q}^1 + 4\tau\sum_{K\in\mathcal{T}_h} \int_K v_h^{1+\sigma_1}f(\cdot, t_{1+\sigma_1})d{\bf x} + 2\tau\Big(\big(a_{0}^{(1)} - a_1^{(1)}\big)\left\|v_h^{1}\right\|_{L^2(\Omega_h)}^2 +a_1^{(1)}\left\|v^0_h\right\|_{L^2(\Omega_h)}^2\Big).
\end{equation}
Substituting \eqref{eq: auxx3} into \eqref{eq: auxx2}, using the coercivity estimate \eqref{eq:A}, Young's inequality and Lemma \ref{lemma:ineq}, we get
\[\begin{aligned}
\mathcal{Q}^{n+1} \leq &\ \mathcal{Q}^1 + 2C_5\tau \sum_{i=2}^n \mathcal{Q}^i + \frac{\tau}{2} \sum_{m=1}^n (\mathcal{Q}^{m+1} + \mathcal{Q}^m) + 8\tau\sum_{m=1}^n\|f(\cdot,t_{m+\sigma_m})\|_{L^2(\Omega_h)}^2 \\
& + 2\Big(C_3\left\|v_h^{1}\right\|_{L^2(\Omega_h)}^2 +C_4\left\|v^0_h\right\|_{L^2(\Omega_h)}^2\Big).
\end{aligned}\]
Thus, for sufficiently small $\tau$ (e.g., $\tau \leq 1$ here), we have
\[\mathcal{Q}^{n+1} \leq C\Big(\mathcal{Q}^1+\left\|v_h^0\right\|_{L^2\left(\Omega_h\right)}^2+\left\|v_h^1\right\|_{L^2\left(\Omega_h\right)}^2+\tau \sum_{m=1}^n\left\|f\left(\cdot, t_{m+\sigma_m}\right)\right\|_{L^2\left(\Omega_h\right)}^2\Big)+C \tau \sum_{m=1}^n \mathcal{Q}^m.\]
The above argument gives the estimate for \(2\le n\le M-1\). 
For \(n=1\), the same bound follows directly from \eqref{eq: auxx3}. 
Therefore, after increasing the generic constant if necessary, we obtain
\[
\mathcal{Q}^{n+1} \!\leq\! 
C\Big(\!\mathcal{Q}^1+\|v_h^0\|_{L^2(\Omega_h)}^2
+\|v_h^1\|_{L^2(\Omega_h)}^2
+\tau\!\! \sum_{m=1}^n
\|f(\cdot,t_{m+\sigma_m})\|_{L^2(\Omega_h)}^2\!\Big)
+C\tau \sum_{m=1}^n \mathcal{Q}^m,
\quad 1\le n\le M-1.
\]
Then, the discrete Gr\"onwall inequality yields
\begin{equation}\label{eq:aux8}
\mathcal{Q}^{n+1} \leq C\Big(\mathcal{Q}^1+\left\|v_h^0\right\|_{L^2\left(\Omega_h\right)}^2+\left\|v_h^1\right\|_{L^2\left(\Omega_h\right)}^2+\tau \sum_{m=1}^n\left\|f\left(\cdot, t_{m+\sigma_m}\right)\right\|_{L^2\left(\Omega_h\right)}^2\Big)\qquad 1\le n\le M-1.
\end{equation}
Given the definition of $\mathcal{Q}^1$ in \eqref{eq:Q}, we are left with the estimate of \(\|\nabla_h u_h^1\|_{L^2(\Omega_h)}\) and \(\|v_h^1\|_{L^2(\Omega_h)}\).
}

\textbf{Step two:} $m = 0$.

To proceed, we set \(\phi = u_h^{1/2}\) and \(\psi = v_h^{1/2}\) in the scheme \eqref{eq:DG_fully1}. Adding the resulting equations, summing over all elements \(K\), and applying the definition of the numerical fluxes in \eqref{flux1}, we then invoke Young's inequality to get
\begin{equation*}
\begin{aligned}
&\frac{1}{2\tau}\Big(\|\nabla_h u_h^1\|_{L^2(\Omega_h)}^2 - \|\nabla_h u_h^0\|_{L^2(\Omega_h)}^2\Big) + \Big(\frac{1}{2\tau} + \frac{1}{2s_0}\Big)\big(\|v_h^1\|_{L^2(\Omega_h)}^2 - \|v_h^0\|_{L^2(\Omega_h)}^2\big) \\
\leq & \ \frac{1}{2s_0} \|v_h^1\|^2_{L^2(\Omega_h)} + \frac{s_0}{4}\|f(\cdot, t_{1/2})\|_{L^2(\Omega_h)}^2 + \frac{1}{2s_0} \|v_h^0\|^2_{L^2(\Omega_h)}.
\end{aligned}
\end{equation*}

Thus, we conclude that
\begin{equation}\label{eq:aux9}
\|\nabla_h u_h^1\|_{L^2(\Omega_h)}^2 + \|v_h^1\|_{L^2(\Omega_h)}^2 \leq \|\nabla_h u_h^0\|_{L^2(\Omega_h)}^2  + \big(1+\frac{2\tau}{s_0}\big) \|v_h^0\|_{L^2(\Omega_h)}^2 + \frac{\tau s_0}{2}\|f(\cdot, t_{1/2})\|_{L^2(\Omega_h)}^2.
\end{equation}

Finally, substituting \eqref{eq:aux9} into \eqref{eq:aux8} completes the proof, since $a_i^{(m)} > 0$ by Lemma \ref{lemma:fractional_approximation} and Lemma \ref{lemma:ineq}.
\end{proof}

\subsection{Error Estimates}\label{sec:error_full}

We now derive error estimates for the fully discrete energy-based DG scheme \eqref{eq:DG_fully1}--\eqref{eq:DG_fully1_m}. We define the numerical errors by 
\begin{equation}\label{eq:def_e}
e_u^m := u^m - u_h^m,\quad e_v^m := v^m - v_h^m,\quad  m = 0,1,\cdots, M.
\end{equation}

{To proceed, we compare the numerical solution
$(u_h^m,v_h^m)$ with suitable projections of the exact solution
$(u^m,v^m)$. Let $\tilde u_h^m$ be the elementwise $H^1$ projection of
$u^m$ onto $U_h^{q_u}$, and let $\tilde v_h^m$ be the elementwise
$L^2$ projection of $v^m$ onto $V_h^{q_v}$. Throughout the error analysis,
the polynomial degrees are chosen from the admissible set
\begin{equation}\label{eq:degree_admissible}
\mathcal{Q}_{\rm adm}
:=
\left\{(q_u,q_v)\in \mathbb{N}_0^2:\ q_u\ge 1,\quad
q_u-2\le q_v\le q_u\right\}.
\end{equation}
Equivalently, the admissible degree pairs consist of the following three
cases:
\begin{equation*}\label{eq:degree_cases}
\begin{aligned}
\mathcal{Q}_{0}
&:=\left\{(q,q): q\ge 1\right\}, &&\text{equal-order approximation},\\
\mathcal{Q}_{1}
&:=\left\{(q,q-1): q\ge 1\right\}, &&\text{one-degree lower approximation for }v_h,\\
\mathcal{Q}_{2}
&:=\left\{(q,q-2): q\ge 2\right\}, &&\text{two-degree lower approximation for }v_h.
\end{aligned}
\end{equation*}
Thus,
\[
\mathcal{Q}_{\rm adm}=\mathcal{Q}_{0}\cup\mathcal{Q}_{1}\cup\mathcal{Q}_{2}.
\]
The restriction \eqref{eq:degree_admissible} is imposed only for the
projection-based error estimate. It is not required for the stability of
the numerical scheme, which holds for general choices of the discrete
spaces $U_h^{q_u}$ and $V_h^{q_v}$. More precisely, for each element $K$, the projections
$(\tilde u_h^m,\tilde v_h^m)\in U_h^{q_u}\times V_h^{q_v}$ are defined by
\begin{equation}\label{def:proj}
\int_{K} \nabla \phi \cdot\nabla \eta_u^m\ d{\bf x} = \int_{K} \psi \eta_v^m \ d{\bf x}= \int_{K} \eta_u^m \ d{\bf x} = 0 \quad \forall \phi \in U_h^{q_u}|_K, \psi \in V_h^{q_v}|_K,
\end{equation}
where
\begin{equation}\label{eq:eta}
 \eta_u^m := \tilde{u}_h^m - u^m,\quad \eta_v^m := \tilde{v}_h^m - v^m.
\end{equation}}

{Before stating the error estimate, we first present a lemma that will be used extensively in the rest of the analysis.}

\begin{lemma}\label{lemma:projection}
{Assume that $\mathcal T_h$ is a shape-regular and quasi-uniform family of affine simplicial or tensor-product meshes, and let $h:=\max_{K\in\mathcal T_h} h_K$. Let $\eta_u^m$ and $\eta_v^m$, $m=0,1,\ldots,M$, be defined by \eqref{eq:eta}, where $\tilde u_h^m$ is the elementwise $H^1$ projection of $u^m$ onto $U_h^{q_u}$, and $\tilde v_h^m$ is the elementwise $L^2$ projection of $v^m$ onto $V_h^{q_v}$ as defined in \eqref{def:proj}. Let
\[
\bar q:=\min\{q_u-1,q_v\},
\qquad
(q_u,q_v)\in\mathcal Q_{\rm adm},
\]
where $\mathcal Q_{\rm adm}$ is defined in \eqref{eq:degree_admissible}. Then, for each element $K\in\mathcal T_h$, there exists a constant $C>0$, independent of $h$ and $m$, such that
\begin{equation*}\label{eq:projection_volume}
\begin{aligned}
\|\eta_v^m\|_{L^2(K)}^2
+
\|\nabla \eta_u^m\|_{L^2(K)}^2
\le
C h^{2\bar q+2}
\left(
|u(\cdot,t_m)|_{H^{\bar q+2}(K)}^2
+
|v(\cdot,t_m)|_{H^{\bar q+1}(K)}^2
\right).
\end{aligned}
\end{equation*}}
{Moreover, let the starred traces, $(\cdot)^\ast$, be defined by the numerical trace operator in \eqref{flux1}. For periodic boundary conditions, boundary faces are identified with their periodic neighboring faces. Let $\omega_K$ denote the set consisting of $K$ and all its face-neighboring elements, including periodic neighbors. Then
\begin{equation*}\label{eq:projection_trace_star}
\|(\eta_v^m)^*\|_{L^2(\partial K)}^2
+
\|(\nabla \eta_u^m)^*\cdot \mathbf n\|_{L^2(\partial K)}^2 
 \le
C h^{2\bar q+1}
\sum_{K'\in\omega_K}
\left(
|u(\cdot,t_m)|_{H^{\bar q+2}(K')}^2
+
|v(\cdot,t_m)|_{H^{\bar q+1}(K')}^2
\right).
\end{equation*}}
{Finally, for the discrete errors between DG solution $(u_u^m, v_h^m)$ and the projection $(\tilde u_h^m, \tilde v_h^m)$
\[
\xi_u^m:=\tilde u_h^m-u_h^m,
\qquad
\xi_v^m:=\tilde v_h^m-v_h^m,
\]
one has
\begin{equation*}\label{eq:projection_inverse_trace}
\|\xi_v^m\|_{L^2(\partial K)}^2
+
\|\nabla \xi_u^m\cdot \mathbf n\|_{L^2(\partial K)}^2
\le
C h^{-1}
\left(
\|\xi_v^m\|_{L^2(K)}^2
+
\|\nabla \xi_u^m\|_{L^2(K)}^2
\right).
\end{equation*}}
\end{lemma}

\begin{proof}
{The volume estimate follows from the standard approximation properties of the elementwise $H^1$ projection and $L^2$ projection, which are consequences of the Bramble--Hilbert lemma (see, e.g., Theorem 4.1.3 of \cite{ciarlet2002finite}). The boundary estimates for the projection errors follow from the scaled trace theorem (see, e.g., Lemma 12.15 of \cite{ern2021finiteI}) combined with these approximation estimates. The final estimate for the discrete errors $\xi_u^m$ and $\xi_v^m$ follows from the inverse trace inequality (see, e.g., Theorem 3.2.6 of \cite{ciarlet2002finite}), since $\xi_u^m$ and $\xi_v^m$ belong to the finite element spaces.}
\end{proof}


We are now ready to present error estimates for the fully discrete energy-based DG scheme \eqref{eq:DG_fully1}--\eqref{eq:DG_fully1_m} with the numerical fluxes \eqref{flux1}.
\begin{truth}\label{them2}
{Assume that $\mathcal T_h$ is a shape-regular and quasi-uniform family of affine simplicial or tensor-product meshes, with $\Omega_h:=\bigcup_{K\in\mathcal T_h}K$ denoting the computational domain. Let
\[
(q_u,q_v)\in\mathcal Q_{\rm adm},
\qquad
\bar q=\min\{q_u-1,q_v\},
\]
where $\mathcal Q_{\rm adm}$ is defined in \eqref{eq:degree_admissible}. Assume further that the variable order $\alpha(t)$ satisfies \eqref{eq:condition_alpha}. Let $(u,v)$ be the exact solution of \eqref{eq:2nd_system} and suppose that
\[
u\in L^{\infty}(0,T;H^{q_u+1}(\mathcal{T}_h))
\cap W^{3,\infty}(0,T;H^1(\mathcal{T}_h))\cap W^{2,\infty}(0,T;H^2(\mathcal{T}_h)),
\]
and
\[
v\in L^\infty(0,T;H^{q_v+1}(\mathcal{T}_h))
\cap W^{3,\infty}(0,T;L^2(\mathcal{T}_h))\cap W^{2,\infty}(0,T;H^1(\mathcal{T}_h)).
\]
Let $(u_h^m, v_h^m)$ be the numerical solution generated by the fully discrete energy-based DG scheme \eqref{eq:DG_fully1}--\eqref{eq:DG_fully1_m}, subject to periodic boundary conditions, the numerical fluxes \eqref{flux1}, and the projected initial data 
\[(u_h^0, v_h^0) = (\tilde{u}_h^0, \tilde{v}_h^0).\] 
Let $(e_u^m, e_v^m)$ be the error defined in \eqref{eq:def_e}. Then, {for sufficiently small time step size $\tau$,} the following error estimate holds 
\begin{equation}\label{eq:error_estimate}
\|\nabla_h e_u^m\|_{L^2(\Omega_h)}^2 + \|e_v^m\|_{L^2(\Omega_h)}^2 \leq C\left(h^{2 r}\tau^4 + h^{2\hat{q}}\right)\qquad \forall m=1,\cdots,M.
\end{equation}
Here, under the general nonnegative flux condition $\gamma,\zeta\ge 0$, one has
\[
\hat q=\bar q,
\qquad
r=-1.
\]
In the strictly dissipative case $\gamma,\zeta>0$, the improved choice
\[
\hat q=\bar q+\tfrac12,
\qquad
r=-\tfrac12
\]
is available. Moreover, the positive constant \(C\) may depend on \(T\), \(q_u\), \(q_v\), \(L_\alpha\), the mesh shape-regularity and quasi-uniformity constants, the bounds of $\alpha(t)$, the fixed mesh-independent numerical flux parameters, and the above spatial and temporal norms of the exact solution, but is independent of the spatial mesh size \(h\), the time step \(\tau\), and the time level \(m\).}
\end{truth}

\begin{proof}
Similar to the stability analysis in the previous section, we proceed with an estimate for the case of $1\leq m\leq M-1$ first (\underline{\textit{step one}}), and then for the case of $m=0$ (\underline{\textit{step two}}).

\textbf{Step one:} $1\leq m\leq M-1$.

Since both the continuous solution $(u({\bf x}, t_m),v({\bf x}, t_m))$ and the DG solution $(u_h^m,v_h^m)$ with $m\geq 1$ satisfy the fully discrete scheme \eqref{eq:DG_fully1_m}, we have
\begin{equation}\label{eq:aux10}
\begin{aligned}
&\int_K \nabla \phi \cdot \nabla \left(\delta_t e_u^{m+\sigma_m} - e_v^{m+\sigma_m}\right)\ d{\bf x} \\
= &\int_{\partial K} \nabla \phi \cdot {\bf n} \big( (e_v^{m+\sigma_m})^\ast - e_v^{m+\sigma_m} \big) dS +  \int_K -\nabla \phi \cdot \nabla E_7^{m+\sigma_m} + \nabla \phi \cdot \nabla E_8^{m+\sigma_m}\ d{\bf x} \\
+ & \int_{\partial K} \nabla \phi \cdot{\bf n} (-E_8^{m+\sigma_m}+E_{9}^{m+\sigma_m}) \ dS,
\end{aligned}
\end{equation}
and
\begin{equation}\label{eq:aux10_1}
\begin{aligned}
&\int_K \psi \delta_t e_v^{m+\sigma_m} + \nabla\psi \cdot \nabla e_u^{m+\sigma_m} + \psi \sum_{i=0}^m  {a_{m-i}^{(m)}} \big(e_v^{i+1} - e_v^i\big) \ d {\bf x}\\
= &\int_{\partial K} \psi (\nabla e_u^{m+\sigma_m})^\ast\cdot {\bf n} \ dS - \int_K \psi E_{10}^{m+\sigma_m} + \nabla \psi \cdot \nabla E_{11}^{m+\sigma_m} + \psi E_{12}^{m+\sigma_m} \ d{\bf x} \\
+ & \int_{\partial K} \psi E_{13}^{m+\sigma_m}\cdot{\bf n}\ dS,
\end{aligned}
\end{equation}
where $\{E_i^{m+\sigma_m}\}_{i=7}^{13}$ denote the local truncation errors:
\begin{equation*}\label{eq:local1}
\begin{aligned}
E_7^{m+\sigma_m} &:= \frac{\partial u(\cdot, t_{m+\sigma_m})}{\partial t} - \delta_t u^{m+\sigma_m}, \quad\quad E_8^{m+\sigma_m} := v(\cdot, t_{m+\sigma}) - v^{m+\sigma_m},\\
E_{9}^{m+\sigma_m} &:= v^\ast(\cdot, t_{m+\sigma_m}) - (v^{m+\sigma_m})^\ast,\quad\quad E_{10}^{m+\sigma_m} := \frac{\partial v(\cdot, t_{m+\sigma_m})}{\partial t} - \delta_t v^{m+\sigma_m},\\
E_{11}^{m+\sigma_m} &:= u(\cdot, t_{m+\sigma_m}) - u^{m+\sigma_m}, \quad\quad\quad\ \  E_{13}^{m+\sigma_m} := \big(\nabla u(\cdot, t_{m+\sigma_m})\big)^\ast - (\nabla u^{m+\sigma_m})^\ast,
\end{aligned}
\end{equation*}
and
\begin{equation*}\label{eq:local2}
\begin{aligned}
E_{12}^{m+\sigma_m} &:= {}_0^C\mathcal{D}_t^{\alpha(t)} v(\cdot, t_{m+\alpha_m}) - \! \sum_{i=0}^m {a_{m-i}^{(m)}}\big(v^{i+1} - v^i\big).
\end{aligned}
\end{equation*}

Substituting $\phi$ by $\xi_u^{m+\sigma_m}$ in \eqref{eq:aux10} and $\psi$ by $\xi_v^{m+\sigma_m}$ in \eqref{eq:aux10_1}, then summing the resulting equations over all elements $K$ and invoking the relations \[e_u^{m+\sigma_m} = \xi_u^{m+\sigma_m} - \eta_u^{m+\sigma_m},\quad e_v^{m+\sigma_m} = \xi_v^{m+\sigma_m} - \eta_v^{m+\sigma_m},\] 
we generate
\begin{equation}\label{eq:aux11}
\begin{aligned}
&\sum_{K\in\mathcal{T}_h}  \int_K \nabla \xi_u^{m+\sigma_m} \cdot \delta_t \nabla \xi_u^{m+\sigma_m} + \xi_v^{m+\sigma_m} \delta_t \xi_v^{m+\sigma_m}  + \xi_v^{m+\sigma_m} \sum_{i=0}^m {a_{m-i}^{(m)}}\big(e_v^{i+1} - e_v^i\big)d{\bf x}\\
= & \!\sum_{K\in\mathcal{T}_h}  \!\!\int_K \!\!\nabla \xi_u^{m+\sigma_m}\! \!\cdot\! \delta_t\nabla  \eta_u^{m+\sigma_m} \!\!+ \xi_v^{m+\sigma_m} \delta_t \eta_v^{m+\sigma_m} \!\!- \!\!\nabla \xi_u^{m+\sigma_m} {\cdot}\nabla \eta_v^{m+\sigma_m} \!\!+ \!\!\nabla \xi_v^{m+\sigma_m} \!\cdot\! \nabla \eta_u^{m+\sigma_m}  \\
  + & \!\sum_{K\in\mathcal{T}_h} \!\int_{\partial K} \!\!\!\!\nabla \xi_u^{m+\sigma_m}\cdot {\bf n} \big( (\xi_v^{m+\sigma_m})^\ast \!\!-\! \xi_v^{m+\sigma_m} \big) \!+ \xi_v^{m+\sigma_m} (\nabla \xi_u^{m+\sigma_m})^\ast\cdot {\bf n} \ dS\\
 - & \!\sum_{K\in\mathcal{T}_h} \!\int_{\partial K} \!\!\nabla \xi_u^{m+\sigma_m}\cdot {\bf n} \big( (\eta_v^{m+\sigma_m})^\ast - \eta_v^{m+\sigma_m} \big) + \xi_v^{m+\sigma_m} (\nabla \eta_u^{m+\sigma_m})^\ast\cdot {\bf n} \ dS + \mathcal{L}_1 + \mathcal{L}_2,
\end{aligned}
\end{equation}
where $\mathcal{L}_1$ and $\mathcal{L}_2$ are terms related to local truncation errors $\{E_i^{m+\sigma_m}\}_{i=7}^{13}$ are given as follows
\[\begin{aligned}
\mathcal{L}_1 &:=  \sum_{K\in\mathcal{T}_h} \!\! \int_K \!\!\nabla \xi_u^{m+\sigma_m} \!\cdot\! \nabla (E_8^{m+\sigma_m}\!\!-\!E_7^{m+\sigma_m} ) \!-\! \xi_v^{m+\sigma_m} (E_{10}^{m+\sigma_m}\!+\!E_{12}^{m+\sigma_m}) \ d{\bf x} \\
& + \sum_{K\in\mathcal{T}_h}\int_{\partial K} \!\nabla \xi_u^{m+\sigma_m} \!\!\cdot{\bf n} (E_{9}^{m+\sigma_m}\!\!-E_8^{m+\sigma_m})\ dS \\
\mathcal{L}_2 &:= \sum_{K\in\mathcal{T}_h} \!\int_K \!-\! \nabla \xi_v^{m+\sigma_m} \!\cdot\! \nabla E_{11}^{m+\sigma_m} d{\bf x} + \!\!\int_{\partial K} \!\!\xi_v^{m+\sigma_m} E_{13}^{m+\sigma_m} \!\!\cdot{\bf n} \ dS.
\end{aligned}\]
Taking an integration by parts for the volume integral involving $\nabla \eta_v^{m+\sigma_m}$ in the second line of \eqref{eq:aux11}, and volume integral involving $\nabla E_{11}^{m+\sigma_m}$ in $\mathcal{L}_2$, and using the definition of the fluxes in \eqref{flux1} together with the projection given in \eqref{def:proj}, we obtain
\begin{equation*}
\sum_{K\in\mathcal{T}_h }  \int_K \nabla \xi_u^{m+\sigma_m} \cdot \delta_t \nabla \xi_u^{m+\sigma_m} + \xi_v^{m+\sigma_m} \delta_t \xi_v^{m+\sigma_m} + \xi_v^{m+\sigma_m} \sum_{i=0}^m {a_{m-i}^{(m)}}  \big(\xi_v^{i+1} - \xi_v^i\big) \ d {\bf x} = \frak{R}^{m+\sigma_m},
\end{equation*}
where the right-hand-side 
\begin{equation*}\label{eq:R}
\begin{aligned}
&\frak{R}^{m+\sigma_m} \\
= &  - \!\!\sum_{F \in \mathcal{F}_h}\!\!\int_{F}[\nabla \xi_u^{m+\sigma_m}](\eta_v^{m+\sigma_m})^\ast  \!+\! [\xi_v^{m+\sigma_m}]\!\cdot\! (\nabla \eta_u^{m+\sigma_m})^\ast \!+\! \gamma \big|[\xi_v^{m+\sigma_m}]\big|^2 \!+\! \zeta [\nabla \xi_u^{m+\sigma_m}]^2 dS \\
& + \!\!\sum_{K\in\mathcal{T}_h} \!\! \int_{\partial K} \!\!\xi_v^{m+\sigma_m} (E_{13}^{m+\sigma_m} - \nabla E_{11}^{m+\sigma_m}) \cdot{\bf n} \ dS + \int_K \xi_v^{m+\sigma_m} \Delta E_{11}^{m+\sigma_m}\ d{\bf x} + \mathcal{L}_1.
\end{aligned}
\end{equation*}
Here, again, \(\mathcal F_h\) denote the set of all interelement faces, with periodic boundary faces identified according to the periodic boundary condition. Then, using the same analysis as in the derivation of Theorem \ref{th:discrete_stability}, we find that
{\begin{equation}\label{eq:aux12_1}
\begin{aligned}
& \mathcal{A}(\xi_v^{m+1}) + \mathcal{A}(\nabla_h \xi_u^{m+1}) + 2\tau\sum_{i=2}^{m+1} a_{m-i+1}^{(m)} \|\xi_v^i\|_{L^2(\Omega_h)}^2 \\
\leq &\ C\Big(\mathcal{A}(\xi_v^1) + \mathcal{A}(\nabla_h \xi_u^1) + \left\|\xi_v^{1}\right\|_{L^2(\Omega_h)}^2+\left\|\xi_v^0\right\|_{L^2(\Omega_h)}^2 + \tau \sum_{i = 1}^{m} \frak{R}^{i+\sigma_i} \Big),
\end{aligned}
\end{equation}}
where $\mathcal{A}(\cdot)$ is defined in \eqref{eq:A}, and  $C$ is a positive constant independent of the spatial mesh size \(h\), the time step size \(\tau\), and the time level \(m\). To complete the estimates, we divide our arguments into two cases: both fluxes parameters $\gamma$ and $\zeta$ in \eqref{flux1} are positive, or at least one of them is zero.

\textit{\textbf{Case 1.}} When the mesh independent flux parameters $\gamma = 0$ or $\zeta = 0$, we use Cauchy--Swartz inequality, Young's inequality, Lemma \ref{lemma:fractional_approximation}, Lemma \ref{lemma:integer_approximation} and Lemma \ref{lemma:projection} to obtain
\begin{equation}\label{eq:aux13}
\tau \sum_{i = 1}^{m} \frak{R}^{i+\sigma_i}
\leq \frac{\tau}{2}\sum_{i = 1}^{m}\Big(  \|\nabla_h \xi_u^{i+\sigma_i}\|_{L^2(\Omega_h)}^2 +\|\xi_v^{i+\sigma_i}\|_{L^2(\Omega_h)}^2 \Big) +C\tau \sum_{i = 1}^{m}\Big( h^{2\bar{q}} + \tau^4 + h^{-2}\tau^4\Big),
\end{equation}
where the scaling $h^{-2}$ comes from the application of the trace inequality. Plugging \eqref{eq:aux13} into \eqref{eq:aux12_1}, and using \eqref{eq:A} together with the fact that $a_{m-i+1}^{(m)} > 0$, we arrive at
\begin{equation}\label{eq:errorm_case1}
\begin{aligned}
\|\xi_v^{m+1}\|_{L^2(\Omega_h)}^2 &\ + \|\nabla_h \xi_u^{m+1}\|_{L^2(\Omega_h)}^2 
\leq  \frac{\tau}{2}\sum_{i = 1}^{m} \Big( \|\nabla \xi_u^{i+\sigma_i}\|_{L^2(\Omega_h)}^2 +\|\xi_v^{i+\sigma_i}\|_{L^2(\Omega_h)}^2 \Big)\\
&+ Ch^{2\bar{q}} + C(\tau^4 + h^{-2}\tau^4) + C\Big(\|\nabla_h \xi_u^1\|_{L^2(\Omega_h)}^2 + \| \xi_v^1\|_{L^2(\Omega_h)}^2\Big).
\end{aligned}
\end{equation}

\textit{\textbf{Case 2.}} When the mesh independent flux parameters $\gamma, \zeta > 0$, one can improve the estimate by abosorbing the boundary terms with flux-dissipating terms. Thanks to Young's inequality, Lemma \ref{lemma:fractional_approximation}, Lemma \ref{lemma:integer_approximation} and Lemma \ref{lemma:projection}, one has
\begin{equation}\label{eq:aux14}
\tau \sum_{i = 1}^{m} \frak{R}^{i+\sigma_i}
\leq  \frac{ \tau}{2}\sum_{i = 1}^{m}\Big(  \|\nabla_h \xi_u^{i+\sigma_i}\|_{L^2(\Omega_h)}^2 +\|\xi_v^{i+\sigma_i}\|_{L^2(\Omega_h)}^2 \Big) + C\tau\sum_{i=1}^m  \Big(h^{2\bar{q}+1} + \tau^4 + h^{-1}\tau^4 \Big).
\end{equation}
Substituting \eqref{eq:aux14} into \eqref{eq:aux12_1} yields
\begin{equation}\label{eq:errorm_case2}
\begin{aligned}
\|\xi_v^{m+1}\|_{L^2(\Omega_h)}^2 &\ + \|\nabla_h \xi_u^{m+1}\|_{L^2(\Omega_h)}^2 
\leq  \frac{\tau}{2}\sum_{i = 1}^{m} \Big( \|\nabla \xi_u^{i+\sigma_i}\|_{L^2(\Omega_h)}^2 +\|\xi_v^{i+\sigma_i}\|_{L^2(\Omega_h)}^2 \Big)\\
&+ Ch^{2\bar{q}+1} + C(\tau^4 + h^{-1}\tau^4) + C\Big(\|\nabla_h \xi_u^1\|_{L^2(\Omega_h)}^2 + \| \xi_v^1\|_{L^2(\Omega_h)}^2\Big).
\end{aligned}
\end{equation}
We are now left to estimate $\|\nabla_h \xi_u^1\|_{L^2(\Omega_h)}^2  + \left\|\xi_v^{1}\right\|_{L^2(\Omega_h)}^2$.

\textbf{Step two:} $m = 0$.
 
 To proceed, we derive the error equation of the fully discretized scheme \eqref{eq:DG_fully1} as follows
\begin{equation}\label{eq:aux15}
\begin{aligned}
&\int_K \nabla \phi \cdot \nabla \left(\frac{e_u^1 - e_u^0}{\tau} - e_v^{1/2}\right) \ d{\bf x} \\
= &\int_{\partial K} \!\!\!\!\nabla \phi \cdot{\bf n} \left((e_v^{1/2})^\ast \!-\!e_v^{1/2}\right)dS + \int_K \nabla \phi \cdot \nabla (-E_0 + E_1)\ d{\bf x} + \int_{\partial K} \!\!\!\nabla \phi \cdot {\bf n} (-E_1 + E_2) \ dS,\\
&\int_K \psi \frac{e_v^1 - e_v^0}{\tau} + \nabla \psi \cdot \nabla e_u^{1/2} + \psi \frac{e_v^1 - e_v^0}{s_0} \ d{\bf x} \\
= &\int_{\partial K} \psi (\nabla e_u^{1/2})^\ast \cdot {\bf n} dS - \int_K \psi E_3 + \nabla \psi\cdot \nabla E_4 + \psi E_5 \ d{\bf x} + \int_{\partial K} \psi E_6 \cdot{\bf n} \ dS,
\end{aligned}
\end{equation}
where local truncation errors $\{E_i\}_{i=0}^6$ are given by:
\begin{equation*}
\begin{aligned}
E_0 &:=\frac{\partial  u(\cdot, t_{1/2})}{\partial t} - \frac{u^1 - u^0}{\tau}, \quad\quad\quad\quad\quad \ E_1 := v(\cdot, t_{1/2}) - \frac{ v^1 +  v^0}{2},\\
E_2 &:=  v^\ast(\cdot, t_{1/2}) - \frac{(v^1)^\ast + (v^0)^\ast}{2}, \quad\quad\quad \ E_3 :=  \frac{\partial v(\cdot, t_{1/2})}{\partial t} - \frac{v^1 - v^0}{\tau},\\
E_4 &:= u(\cdot, t_{1/2}) - \frac{u^1 + u^0}{2},\quad\quad\quad\quad\quad\quad E_5 := {}_0^C\mathcal{D}_t^{\alpha(t)} v(\cdot, t_{1/2}) - \frac{v^1 - v^0}{s_0},\\
\end{aligned}
\end{equation*}
and
\begin{equation*}
\begin{aligned}
E_6 &:=  \big(\nabla u(\cdot, t_{1/2})\big)^\ast - \frac{(\nabla u^1)^\ast + (\nabla u^0)^\ast}{2}.
\end{aligned}
\end{equation*}

Similar to the case where $1\leq m\leq M-1$, we choose $\phi = \xi_u^{1/2}$ and $\psi = \xi_v^{1/2}$ in \eqref{eq:aux15}, add the resulting equations, and sum over all elements $K$ to generate
\begin{equation}\label{eq:aux16}
\begin{aligned}
&\sum_{K\in\mathcal{T}_h}  \int_K \frac{1}{\tau}\nabla \xi_u^{1/2} \cdot \nabla (\xi_u^1 - \xi_u^0) - \frac{1}{\tau}\nabla  \xi_u^{1/2} \cdot \nabla (\eta_u^1 - \eta_u^0) + \nabla \xi_u^{1/2}\cdot \nabla \eta_v^{1/2} \ d{\bf x}\\
&+\sum_{K\in\mathcal{T}_h}  \int_K\left(\frac{1}{\tau}+\frac{1}{s_0}\right) \xi_v^{1/2}(\xi_v^{1} - \xi_v^0) -\left(\frac{1}{\tau} +\frac{1}{s_0}\right)\xi_v^{1/2}(\eta_v^1 - \eta_v^0) - \nabla \xi_v^{1/2}\cdot\nabla \eta_u^{1/2} \ d{\bf x} \\
= & \sum_{K\in\mathcal{T}_h} \int_{\partial K} \nabla \xi_u^{1/2} \cdot{\bf n} \left((\xi_v^{1/2})^\ast -\xi_v^{1/2}\right) + \xi_v^{1/2} (\nabla \xi_u^{1/2})^\ast \cdot{\bf n}\ dS\\
 - &\sum_{K\in\mathcal{T}_h} \int_{\partial K} \nabla \xi_u^{1/2} \cdot{\bf n} \left((\eta_v^{1/2})^\ast -\eta_v^{1/2}\right) + \xi_v^{1/2} (\nabla \eta_u^{1/2})^\ast \cdot{\bf n} \ dS +\bar{\mathcal{L}}_1 + \bar{\mathcal{L}}_2
\end{aligned}
\end{equation}
where we have used the relation $e_u^m = \xi_u^m - \eta_u^m$ and $e_v^m = \xi_v^m - \eta_v^m, m=0,1$, and
\[\begin{aligned}
\bar{\mathcal{L}}_1 &:=  \sum_{K\in\mathcal{T}_h}  \int_K \nabla \xi_u^{1/2} \cdot \nabla (E_1-E_0) - \xi_v^{1/2} (E_3+ E_5 )\ d{\bf x} + \int_{\partial K} \!\!\!\nabla \xi_u^{1/2} \cdot {\bf n} (E_2-E_1) \ dS, \\
\bar{\mathcal{L}}_2 &:= \sum_{K\in\mathcal{T}_h}  \int_K -\nabla \xi_v^{1/2}\cdot \nabla E_4 \ d{\bf x} + \int_{\partial K} \xi_v^{1/2} E_6 \cdot{\bf n} \ dS.
\end{aligned}\]
Taking an integration by parts for the volume integral involving $\nabla \xi_v^{1/2}$ in the second line of \eqref{eq:aux16}, and the volume integral involving $\nabla E_4$ in $\bar{\mathcal{L}}_2$, we then simplify \eqref{eq:aux16}, by using the projection defined in \eqref{def:proj}, to
\begin{equation}\label{eq:aux30}
\begin{aligned}
&\sum_{K\in\mathcal{T}_h}  \int_K \frac{1}{\tau}\nabla \xi_u^{1} \cdot \nabla \xi_u^1 + \left(\frac{1}{\tau}+\frac{1}{s_0}\right) (\xi_v^1)^2\ d{\bf x} \\
= & \sum_{F\in\mathcal{F}_h}\int_F \!\!-\zeta |[\nabla \xi_u^1]|^2 \!-\! \gamma |[\xi_v^1]|^2 \!-\! [\nabla \xi_u^1] (\eta_v^1)^\ast \!-\! [\xi_v^1] \cdot (\nabla \eta_u^1)^\ast \!-\! [\nabla \xi_u^1] (\eta_v^0)^\ast \!-\! [\xi_v^1] \cdot (\nabla \eta_u^0)^\ast  \ dS\\
  + & \sum_{K\in\mathcal{T}_h}  \int_K  \xi_v^{1/2}\Delta E_4 \ d{\bf x} + \int_{\partial K} \xi_v^{1/2} E_6 \cdot{\bf n}  - \xi_v^{1/2}\nabla E_4\cdot{\bf n}\ dS + \bar{\mathcal{L}}_1.
\end{aligned}
\end{equation}
Analogue to the case of $1\leq m\leq M-1$, we divide our discussion into two cases: 

\textit{\textbf{Case 1.}} When $\gamma=0$ or $\zeta=0$, invoking the Cauchy-Schwarz inequality, the Young's inequality, and Lemma \ref{lemma:projection}, \eqref{eq:aux30} yields
\begin{equation}\label{eq:error1_case1}
\begin{aligned}
&\|\nabla_h \xi_u^1\|_{L^2(\Omega_h)}^2 + \| \xi_v^1\|_{L^2(\Omega_h)}^2 + \frac{\tau}{s_0} \| \xi_v^1\|_{L^2(\Omega_h)}^2\\
\leq & \ \frac{1}{2} \Big(\|\nabla_h \xi_u^1\|_{L^2(\Omega_h)}^2 + \| \xi_v^1\|_{L^2(\Omega_h)}^2 \Big) + C\left(\tau^2h^{2\bar{q}} + \tau^{6-2\alpha_{1/2}} + h^{-2}\tau^{6}\right).
\end{aligned}
\end{equation}

\textit{\textbf{Case 2.}} When $\gamma, \zeta > 0$, one can further find that
\begin{equation}\label{eq:error1_case2}
\begin{aligned}
&\|\nabla_h \xi_u^1\|_{L^2(\Omega_h)}^2 + \| \xi_v^1\|_{L^2(\Omega_h)}^2 + \frac{\tau}{s_0} \| \xi_v^1\|_{L^2(\Omega_h)}^2\\
\leq & \ \frac{1}{2} \Big(\|\nabla_h \xi_u^1\|_{L^2(\Omega_h)}^2 + \| \xi_v^1\|_{L^2(\Omega_h)}^2 \Big)+C\left(\tau^2 h^{2\bar{q}+1} + \tau^{6-2\alpha_{1/2}} + h^{-1}\tau^6\right).
\end{aligned}
\end{equation}

Finally, combining \eqref{eq:errorm_case1} with \eqref{eq:error1_case1}, and \eqref{eq:errorm_case2} with \eqref{eq:error1_case2}, using the fact that 
\[\|\nabla_h \xi_u^{i+\sigma_i}\|_{L^2(\Omega_h)}^2 \leq \|\nabla_h \xi_u^{i+1}\|_{L^2(\Omega_h)}^2 + \|\nabla \xi_u^{i}\|_{L^2(\Omega_h)}^2, \quad \|\xi_v^{i+\sigma_i}\|_{L^2(\Omega_h)}^2 \leq \|\xi_v^{i+1}\|_{L^2(\Omega_h)}^2 + \|\xi_v^{i}\|_{L^2(\Omega_h)}^2,\]
and discrete Gr\"ownwall's inequality together with sufficiently small $\tau$, we get \eqref{eq:error_estimate}.

\end{proof}

\begin{remark}\label{remark:optimal}
Building on the discussion in \cite{appelo2015new}, an optimal error estimate of \(\mathcal{O}(h^{q}, h^r\tau^2)\) can be obtained when \(q_u = q\) and \(q_v = q-1\) for both one-dimensional and multidimensional problems with Cartesian elements. This is achieved by constructing \((\tilde{u}_h^m, \tilde{v}_h^m)\) such that the boundary terms involving \((\eta_v^{m+\sigma_m})^\ast\) and \((\nabla\eta_u^{m+\sigma_m})^\ast\) in \eqref{eq:R} vanish, provided that the flux parameters in \eqref{flux1} satisfy \(\theta(1 - \theta) = \gamma\zeta\). Furthermore, \(r = -1/2\) if both \(\gamma\) and \(\zeta\) are positive, whereas \(r = -1\) if at least one of them is zero.  
\end{remark}

The reduction in the temporal error estimate to \(\mathcal{O}(h^r\tau^2)\) appears to be related to the estimation technique used in the analysis. In the smooth manufactured-solution tests reported in Section \ref{sec:numerical}, the observed temporal convergence remains second order even when $h \approx \tau$ and $\gamma\zeta =0$. A sharper theoretical analysis that removes this technical loss will be pursued in future work.

{Another important issue is the regularity of the exact solution near the initial time. Solutions of time-fractional evolution equations with Caputo derivatives often exhibit weak initial singularities, whereas Theorem \ref{them2} is proved under smooth-in-time regularity assumptions. The present analysis therefore does not fully address the reduced temporal regularity that may occur near \(t=0\). To illustrate this issue numerically, Section \ref{sec:numerical} includes an additional test with a weak initial singularity. This experiment is performed on uniform temporal meshes, consistent with the fully discrete scheme analyzed in this paper, and is intended to document the effect of reduced regularity on the observed convergence behavior. A systematic treatment using nonuniform or graded temporal meshes is left for future work.}

\section{Numerical Simulations}\label{sec:numerical}

In this section, we conduct numerical experiments to validate the stability and convergence of the fully discrete energy-based DG scheme introduced in Section \ref{sec:full_discrete}. Utilizing a standard modal basis formulation, we set the approximation degrees \(q_u = q\) for \(u\) and \(q_v = q-1\) for \(v\) to ensure optimal convergence, as outlined in Remark \ref{remark:optimal}. Numerical errors in $u$ at time $t_m = m\tau$ are reported in the \(L^2\) norm, computed as follows:
\begin{equation*}\label{eq:errors}
E_u(t_m;h,\tau) := \left\| u({\bf x}, t_m) - u_h^m({\bf x})\right\|_{L^2(\Omega_h)} = \Big(\sum_{K\in \mathcal{T}_h} \sum_{j=0}^q w_j^K\big(u({\bf x}_j^K, t_m) - u_h^m({\bf x}_j^K)\big)^2 \Big)^{1/2},
\end{equation*}
where ${\bf x}_j^K$ denotes the Gauss quadrature nodes in element $K$, and $w_j^K$ are the corresponding weights.
Additionally, the root of the nonlinear function \eqref{eq:F} is determined using the Newton method with an initial guess of \( \sigma = \frac{3}{4} \) and a tolerance of \( 10^{-15} \).  

\subsection{Example in 1D}
Consider the variable-order time-fractional wave equation:
\begin{equation}\label{eq:prob1}
 \frac{\partial^2 u(x, t)}{\partial t^2}+{}_0^C \mathcal{D}_t^{\alpha(t)+1} u(x, t)
 =\frac{\partial^2 u(x, t)}{\partial x^2}+f(x, t), \quad(x, t) \in(0, 2\pi) \times(0,1),
\end{equation}
where $\alpha(t)$ satisfies the condition \eqref{eq:condition_alpha}. The exact solution is chosen in the seperated form
\begin{equation}\label{eq:sol_form}
u(x, t) = G(t)\Phi(x).
\end{equation}
The source term $f(x,t)$, together with the initial and boundary data, is determined by substituting \eqref{eq:sol_form} into \eqref{eq:prob1}. Periodic boundary conditions are imposed on $(0,2\pi)$.

\paragraph{Convergence test.} To investigate the convergence behavior of the proposed scheme
\eqref{eq:DG_fully1}--\eqref{eq:DG_fully1_m}, we first consider a sufficiently
regular manufactured solution of the form \eqref{eq:sol_form}, with
{\begin{equation}\label{eq:Phi}
G(t)=t^2+t^{7/2}+\frac12 t^4,
\qquad
\Phi(x)=\Big(1+\frac14\cos x+\frac15\sin 2x\Big)\sin x.
\end{equation}}
The corresponding source term $f(x,t)$ is computed analytically from
\eqref{eq:prob1}. Since $G(0)=G'(0)=0$, the associated initial conditions are $u(x,0)=0, u_t(x,0)=0$. The spatial interval is uniformly partitioned with nodes $x_j=jh$, $j=0,\ldots,N$, where $h=2\pi/N$. Similarly, the temporal interval is uniformly partitioned with nodes $t_m=m\tau$, $m=0,\ldots,M$, where $\tau=1/M$. In this test, we take the flux parameters $\theta=\gamma=\zeta=0$. The polynomial degrees for the approximation spaces of $u_h$ and $v_h$ are chosen as $q_u = q\in\{1,2,3,4\}$ and $q_v = q-1$, respectively. We consider the following variable orders:
\begin{equation}\label{eq:choice_of_beta}
{\alpha(t)\in
\big\{
0.1+0.8e^{-t},\quad 0.9-0.5t^2,\quad (2+\sin t)/4
\big\}.}
\end{equation}

Table \ref{tab:convergence_1d_space_u_al} reports the spatial convergence results. In this test, we vary the number of elements as $N=10,20,40,80$ and fix the number of time steps at $M=5\times 10^4$, so that the temporal discretization error is negligible compared with the spatial discretization error. The results show that the proposed scheme achieves robust and optimal $(q+1)$-st order accuracy for $u$ for all tested variable orders $\alpha(t)$ in \eqref{eq:choice_of_beta}.

\begin{table}[h!]
 	\footnotesize
 \begin{center}
 		\scalebox{1.0}{
 		\begin{tabular}{c| c c c c c c c c c c c}
 				\hline
 				 ~ & ~ & $\alpha(t)$ & $\vert$ &$0.1+0.8e^{-t}$ & ~ & ~ & $0.9-0.5t^2$ & ~ &  ~& $(2+\sin t)/4$ & ~ \\
 				\cline{2-3} \cline{5-6} \cline{8-9} \cline{11-12} 
      ~ & $q$ & $N$ & $\vert$ & $E_u(h,\tau)$ & $\text{order}_h$ & ~ & $E_u(h,\tau)$  & $\text{order}_h$ & ~ & $E_u(h,\tau)$  & $\text{order}_h$  \\
 				\cline{2-12}
 ~ & \cellcolor{gray!15} ~ & 10 & $\vert$ & 3.38e-01& \cellcolor{gray!15}--  & ~ & 3.38e-01& \cellcolor{gray!15}--   & ~ & 3.38e-01& \cellcolor{gray!15}--  \\
 ~ & \cellcolor{gray!15} ~ & 20 & $\vert$ & 8.96e-02& \cellcolor{gray!15}1.92& ~ & 8.95e-02& \cellcolor{gray!15}1.92 & ~ &8.95e-02 & \cellcolor{gray!15}1.92  \\
 ~ & \cellcolor{gray!15}1  & 40 & $\vert$ & 2.27e-02& \cellcolor{gray!15}1.98& ~ & 2.27e-02& \cellcolor{gray!15}1.98 & ~ &2.27e-02  & \cellcolor{gray!15}1.98 \\
 ~ & \cellcolor{gray!15} ~ & 80 & $\vert$ & 5.70e-03& \cellcolor{gray!15}1.99& ~ & 5.70e-03& \cellcolor{gray!15}1.99 & ~ &5.70e-03  & \cellcolor{gray!15}1.99  \\
 				   ~ & ~ & ~  & ~ & ~ & ~ & ~ & ~ & ~ & ~ & ~ & ~ \\
 ~ & \cellcolor{cyan!5} ~  & 10  & $\vert$ & 4.46e-02& \cellcolor{cyan!5}--  & ~ & 4.46e-02& \cellcolor{cyan!5}--   & ~   &4.39e-02 & \cellcolor{cyan!5}--  \\
 ~ & \cellcolor{cyan!5} ~  & 20 & $\vert$ & 6.39e-03& \cellcolor{cyan!5}2.80 & ~ & 6.40e-03& \cellcolor{cyan!5}2.80 & ~ &6.37e-03 & \cellcolor{cyan!5}2.78  \\
 $u$ & \cellcolor{cyan!5}2   & 40 & $\vert$ & 8.27e-04& \cellcolor{cyan!5}2.95& ~ & 8.27e-04& \cellcolor{cyan!5}2.95 & ~ &8.26e-04  & \cellcolor{cyan!5}2.94 \\
 ~ & \cellcolor{cyan!5} ~  & 80 & $\vert$ & 1.04e-04& \cellcolor{cyan!5}2.99& ~ & 1.04e-04& \cellcolor{cyan!5}2.99 & ~ &1.04e-04  & \cellcolor{cyan!5}2.99  \\
 				~ & ~ & ~ & ~  & ~ & ~ & ~ & ~ & ~ & ~ & ~ & ~\\
 ~ & \cellcolor{blue!5} ~  & 10  & $\vert$ & 3.06e-03&  \cellcolor{blue!5}--  & ~ & 3.04e-03& \cellcolor{blue!5}--  & ~  &3.05e-03 & \cellcolor{blue!5}--  \\
 ~ & \cellcolor{blue!5} ~  & 20 & $\vert$ & 2.18e-04&  \cellcolor{blue!5}3.81& ~ & 2.18e-04& \cellcolor{blue!5}3.80 & ~ &2.17e-04 & \cellcolor{blue!5}3.81  \\
 ~ & \cellcolor{blue!5}3   & 40 & $\vert$ & 1.42e-05&  \cellcolor{blue!5}3.94& ~ & 1.42e-05& \cellcolor{blue!5}3.94 & ~ &1.42e-05  & \cellcolor{blue!5}3.94 \\
 ~ & \cellcolor{blue!5} ~  & 80 & $\vert$ & 8.95e-07&  \cellcolor{blue!5}3.98& ~ & 8.94e-07& \cellcolor{blue!5}3.98 & ~ &8.94e-07  & \cellcolor{blue!5}3.98 \\
 ~ & ~ & ~ & ~  & ~ & ~ & ~ & ~ & ~ & ~ & ~ & ~\\
 ~ & \cellcolor{magenta!5} ~  & 10  & $\vert$ & 2.78e-04&  \cellcolor{magenta!5}--  & ~ & 2.79e-04& \cellcolor{magenta!5}--  & ~  &2.80e-04 & \cellcolor{magenta!5}--  \\
 ~ & \cellcolor{magenta!5} ~  & 20 & $\vert$ & 8.88e-06&  \cellcolor{magenta!5}4.97& ~ & 8.87e-06& \cellcolor{magenta!5}4.98 & ~ &8.89e-06 & \cellcolor{magenta!5}4.98  \\
 ~ & \cellcolor{magenta!5}4   & 40 & $\vert$ & 2.79e-07&  \cellcolor{magenta!5}4.99& ~ & 2.79e-07& \cellcolor{magenta!5}4.99 & ~ &2.80e-07  & \cellcolor{magenta!5}4.99 \\
 ~ & \cellcolor{magenta!5} ~  & 80 & $\vert$ & 8.99e-09&  \cellcolor{magenta!5}4.96 &  & 8.88e-09& \cellcolor{magenta!5}4.98 & ~ &9.00e-09  & \cellcolor{magenta!5}4.96 \\
 				\hline
 \end{tabular}
 		}
 \end{center}
\caption{\scriptsize{$L^2$ errors and corresponding convergence rates for $u$ in the problem \eqref{eq:prob1}, using $\mathcal{P}^q$ polynomials with $q=1,2,3,4$. The time step size is $\tau=2\times10^{-5}$ with the terminal time $T=1$, and the spatial interval is divided into $N$ uniform cells. The numerical fluxes parameters are set to $\theta = \gamma = \zeta = 0$. These results show that optimal convergence is robust across different orders of $\alpha(t)$. We also observe optimal convergence for $v$. These results are omitted here for brevity.}}\label{tab:convergence_1d_space_u_al}
 \end{table}

Table \ref{tab:convergence_1d_time_al} presents the temporal convergence results. Here, we fix the number of spatial elements at $N=200$ and choose $q_u=5$ so that the spatial discretization error is negligible compared with the temporal discretization error. We then vary the number of time steps as $M=100,200,300,400$. For all variable orders listed in \eqref{eq:choice_of_beta}, second-order convergence in the $L^2$ norm is observed, which is consistent with the theoretical prediction.

\begin{table}[h!]
 	\footnotesize
 \begin{center}
 		\scalebox{1.0}{
 		\begin{tabular}{c| c c c c c c c c c c}
 				\hline
 				~ & $\alpha(t)$ & $\vert$ & $0.1+0.8e^{-t}$  & ~ & ~ & $0.9-0.5t^2$ & ~ &  ~& $(2+\sin t)/4$ & ~ \\
 				\cline{2-2} \cline{4-5} \cline{6-8} \cline{9-11}
 				~ & $M$ & $\vert$ & $E_u(h,\tau)$ & $\text{order}_\tau$ & ~ &$E_u(h,\tau)$ & $\text{order}_\tau$ & ~ & $E_u(h,\tau)$ & $\text{order}_\tau$  \\
 				\cline{2-11}
~ & 100 & $\vert$ &5.02e-04&   --& ~ & 3.66e-04&   -- & ~ &5.00e-04 &   --  \\
$u$ & 200 & $\vert$ & 1.27e-04& 2.00& ~ & 9.24e-05& 1.98 & ~ &1.27e-04 & 1.97  \\
~ & 400 & $\vert$ &3.18e-05& 2.00& ~ & 2.33e-05& 1.99 & ~ &3.23e-05 & 1.98  \\
~ & 800 & $\vert$ &7.99e-06& 2.00& ~ & 5.86e-06& 1.99 & ~ &8.13e-06 & 1.98  \\
 				\hline
 \end{tabular}
 		}
 \end{center}
\caption{\scriptsize{$L^2$ errors and corresponding convergence rates for $u$ in problem \eqref{eq:prob1}, using $\mathcal{P}^5$ polynomials with a spatial mesh size of $h=2\pi/200$. The time interval is divided into $M$ uniform cells, with the end time set to $T=1$, and the numerical fluxes parameters are $\theta=\gamma=\zeta=0$. These results confirm that the second-order convergence is robust to the fractional order $\alpha(t)$. We also observe 2nd-order convergence for $v$. These results are omitted here for brevity.}}\label{tab:convergence_1d_time_al}
 \end{table}

Finally, Table \ref{tab:convergence_1d_time_al_simu} reports the results when the spatial and temporal mesh sizes are refined simultaneously. In this test, we take $q_v=1$ and choose $N$ and $M$ such that $h \approx \tau$. A consistent second-order convergence rate is observed. We note, however, that Theorem \ref{them2} and Remark \ref{remark:optimal} rigorously establish only the bound $\mathcal{O}(h^{-1}\tau^2)$ for the contribution of the temporal discretization error when $\gamma \zeta = 0$. The factor $h^{-1}$ arises from the trace inequality used to estimate the boundary integrals associated with the local truncation error terms. Removing this additional factor and establishing the observed second-order convergence rigorously is left for future work.
 \begin{table}[h!]
 	\footnotesize
 \begin{center}
 		\scalebox{1.0}{
 		\begin{tabular}{c| c c c c c c c c c c c}
 				\hline
 				~ & ~ & $\alpha(t)$ & $\vert$ & $0.1+0.8e^{-t}$  & ~ & ~ & $0.9-0.5t^2$ & ~ &  ~& $(2+\sin t)/4$ & ~ \\
 				\cline{2-3} \cline{5-6} \cline{7-9} \cline{10-12}
 				~ & $M$ & $N$ & $\vert$ & $E_v(h,\tau)$ & $\text{order}$ & ~ &$E_v(h,\tau)$ & $\text{order}$ & ~ & $E_v(h,\tau)$ & $\text{order}$  \\
 				\cline{2-12}
 & 100 & 100 & $\vert$ &4.59e-03&   --& ~ & 4.60e-03&   -- & ~ &4.60e-03 &   --  \\
$v$ & 200 & 200 & $\vert$ & 1.15e-03& 2.00& ~ & 1.15e-03& 2.00 & ~ &1.15e-03 & 2.00  \\
 & 400 & 400 &$\vert$ &2.87e-04& 2.00& ~ & 2.87e-04& 2.00 & ~ &2.87e-04 & 2.00  \\
          ~ & 800 & 800 & $\vert$ &7.18e-05& 2.00& ~ & 7.18e-05& 2.00 & ~ &7.18e-05 & 2.00  \\
 				\hline
 \end{tabular}
 		}
 \end{center}
\caption{\scriptsize{$L^2$ errors and corresponding convergence rates for $v$ in the problem \eqref{eq:prob1} using $\mathcal{P}^1$ polynomials. The spatial interval is divided into $N$ uniform cells, and the time interval is divided into $M$ uniform cells, with a final time of $T=1$, and the numerical fluxes parameters are $\theta=\gamma=\zeta=0$. We observe a consistent 2nd-order convergence across all choices of $\alpha(t)$.}}\label{tab:convergence_1d_time_al_simu}
 \end{table}

{\paragraph{Weak initial singularity.}  We next examine the performance of the proposed scheme \eqref{eq:DG_fully1}--\eqref{eq:DG_fully1_m} for a solution with limited regularity near the initial time. This example is not covered by the smooth-in-time regularity assumptions used in the error analysis, but is included to investigate the behavior of the method when an initial weak singularity is present. Let
\begin{equation}\label{weak_singular}
    u(x,t)=t^{3/2} \Phi(x),
\end{equation}
where \(\Phi(x)\) is the function defined in \eqref{eq:Phi}. The corresponding source term is obtained by substituting this exact solution into the governing equation. The initial data are $u(x,0)=0, u_t(x,0)=0$. Although \(u\) and \(u_t\) are continuous at \(t=0\), the second time derivative satisfies $u_{tt}(x,t)= \frac34 t^{-1/2}\Phi(x)$, which is singular as \(t\to0^+\). Therefore, this example provides a numerical test of the scheme under reduced temporal regularity.}

{Since the singularity occurs in time, we fix a sufficiently fine spatial mesh with \(h=2\pi/200\) and use \(\mathcal P^5\) elements, as in the spatial convergence test reported in Table \ref{tab:convergence_1d_space_u_al}. We then study the temporal convergence on uniform time grids with $M$ time steps. In particular, we report the maximum-in-time $L^2$ errors} 
\[{
    E_u^{\text{max}}:=\max_{1\le m\le M}\|u(x,t_m)-u_h^m(x)\|_{L^2(\Omega_h)},
    \qquad
    E_v^{\text{max}}:=\max_{1\le m\le M}\|u_t(x,t_m)-v_h^m(x)\|_{L^2(\Omega_h)}.}
\]
{The results are presented in Table \ref{tab:convergence_1d_time_weak_singularity} for variable order $\alpha(t) = \{0.1+0.8e^{-t}, (2+\sin t)/4 \}$ and fix flux parameters $\theta=\gamma=\zeta = 0$. The table shows a clear reduction in the observed temporal convergence rate, which is consistent with the weak initial singularity of the exact solution. We use uniform time steps in this test because the fully discrete scheme analyzed in this paper is formulated on uniform temporal meshes. The purpose of this example is not to recover the optimal rate by mesh grading, but rather to document the effect of weak initial-time singularity on the observed convergence behavior of the proposed method. A thorough study of graded temporal meshes for recovering higher-order convergence is left for future work.}

\begin{table}[h!]
 	\footnotesize
 \begin{center}
 		\scalebox{0.965}{
 		\begin{tabular}{c c c c c c c | c c c c c }
        \hline
        & $\vert$ & & $u$ & & & & & $v$ & & & \\
 				\hline
 				$\alpha(t)$ & $\vert$ & $0.1+0.8e^{-t}$  & ~ & ~ & $(2+\sin t)/4$ &  ~& $0.1+0.8e^{-t}$ & ~ & & $(2+\sin t)/4$ &   \\
 				\cline{1-1} \cline{3-4} \cline{5-7} \cline{8-12}
 				$M$ & $\vert$ & $E_u^{\text{max}}$ & $\text{order}_\tau$&~ &$E_u^{\text{max}}$ & $\text{order}_\tau$  & $E_v^{\text{max}}$ & $\text{order}_\tau$ & ~ & $E_v^{\text{max}}$ & $\text{order}_\tau$ \\
 				\cline{1-12}
 100 & $\vert$ &4.48e-02&   --& ~ & 3.33e-02&   -- &6.93e-04 &   --  & ~ &7.11e-02 &   -- \\
 200 & $\vert$ & 3.11e-02& 0.53& ~ & 2.34e-02& 0.51 &4.94e-05 & 0.49  & ~ &5.19e-02 & 0.46\\
400 & $\vert$ &2.15e-02& 0.53& ~ & 1.65e-02& 0.50 &3.52e-06 & 0.49 & ~ &3.75e-02 & 0.47 \\
800 & $\vert$ &1.49e-02& 0.53& ~ & 1.17e-02& 0.50  &2.50e-06 & 0.49 & ~ &2.70e-02 & 0.48\\
 				\hline
 \end{tabular}
 		}
 \end{center}
\caption{
\scriptsize{Maximum-in-time $L^2$ errors and corresponding convergence rates for $u$ and $v$ for the weakly singular solution \eqref{weak_singular}. The computations use \(\mathcal P^5\) polynomials with spatial mesh size \(h=2\pi/200\). The time interval \([0,1]\) is divided into \(M\) uniform time cells, and the numerical flux parameters are set to \(\theta=\gamma=\zeta=0\). The results show a half-order temporal convergence in the presence of the weak initial singularity.}}\label{tab:convergence_1d_time_weak_singularity}
 \end{table}

{To further examine where the most significant errors occur, we report the
time histories of
\[
    \|v(t_m,x) - v_h^m(x)\|_{L^2(\Omega_h)}
    \qquad \text{and} \qquad
    \|(v_h^m(x)-v_h^{m-1}(x))/\tau\|_{L^2(\Omega_h)},
    \qquad m=1,2,\cdots,M=800,
\]
in Figure \ref{fig:time_history_weak_singularity}. The first quantity measures the error in \(v=u_t\), while the second quantity provides a discrete approximation of the magnitude of \(v_t\). The results show that the largest errors are concentrated near the initial time, which is consistent with the singular behavior of the second-order time derivative \(u_{tt}\).}

\begin{figure}
\centering
\includegraphics[width=0.43\textwidth]{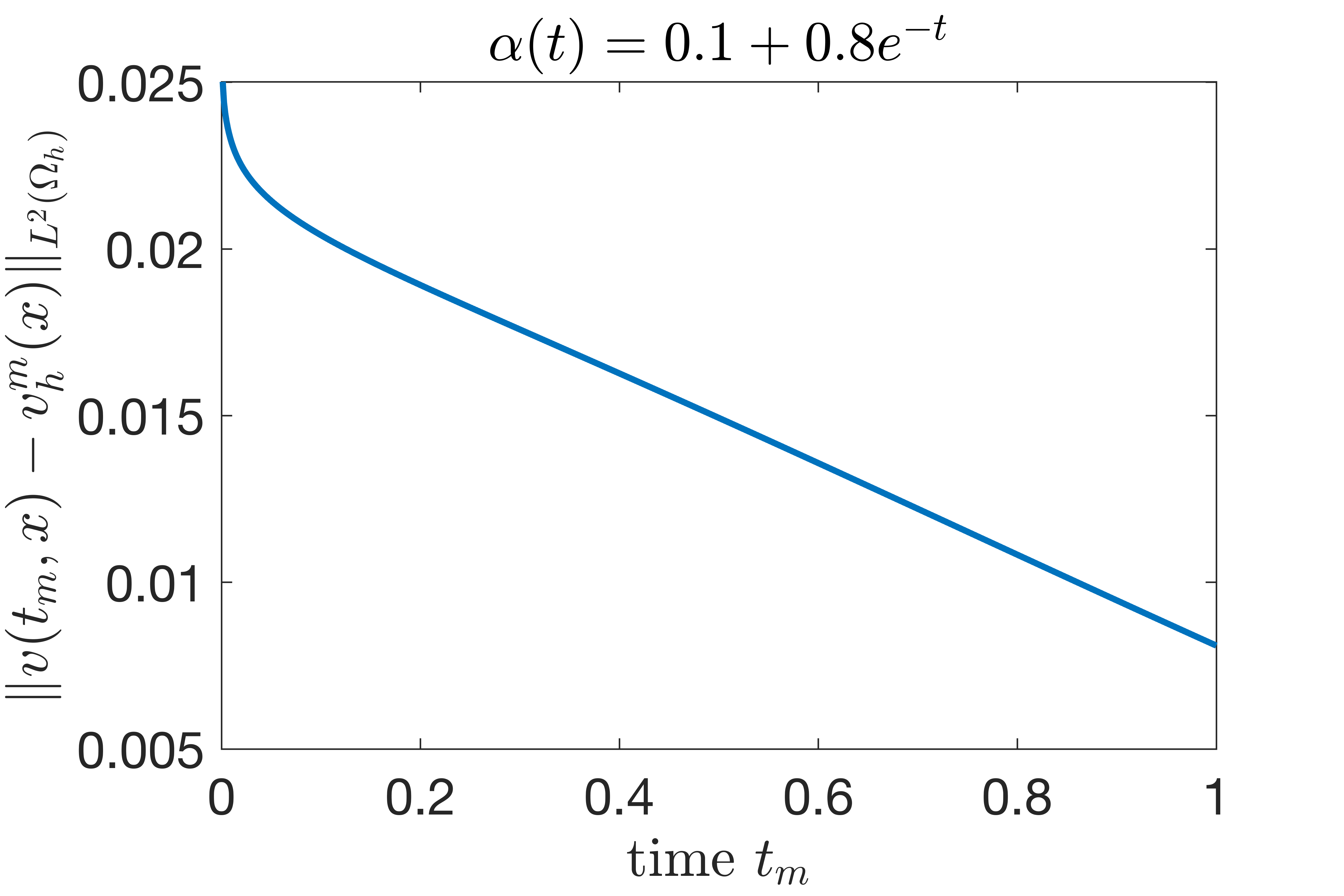} \quad  
\includegraphics[width=0.43\textwidth]{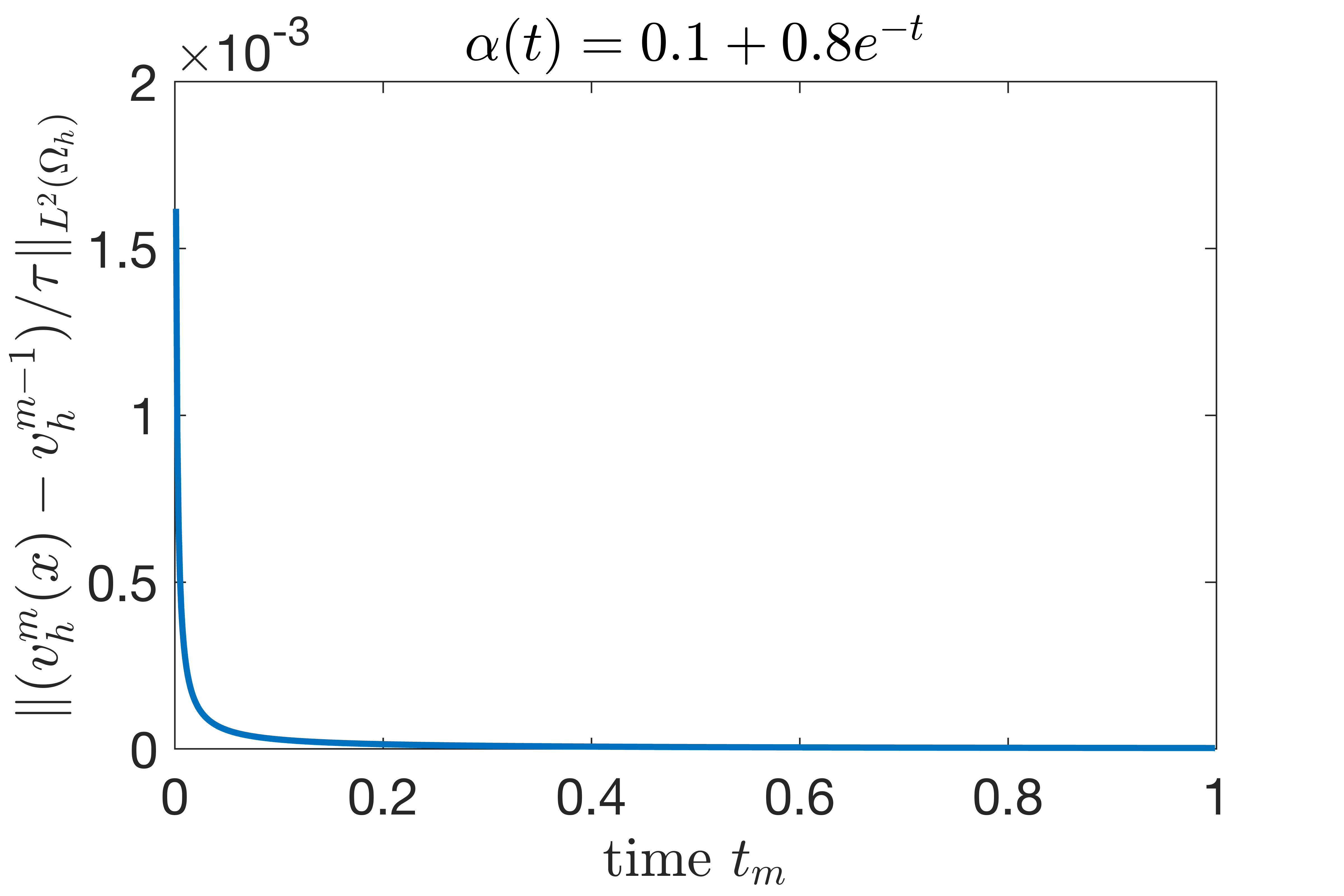} \\
\includegraphics[width=0.43\textwidth]{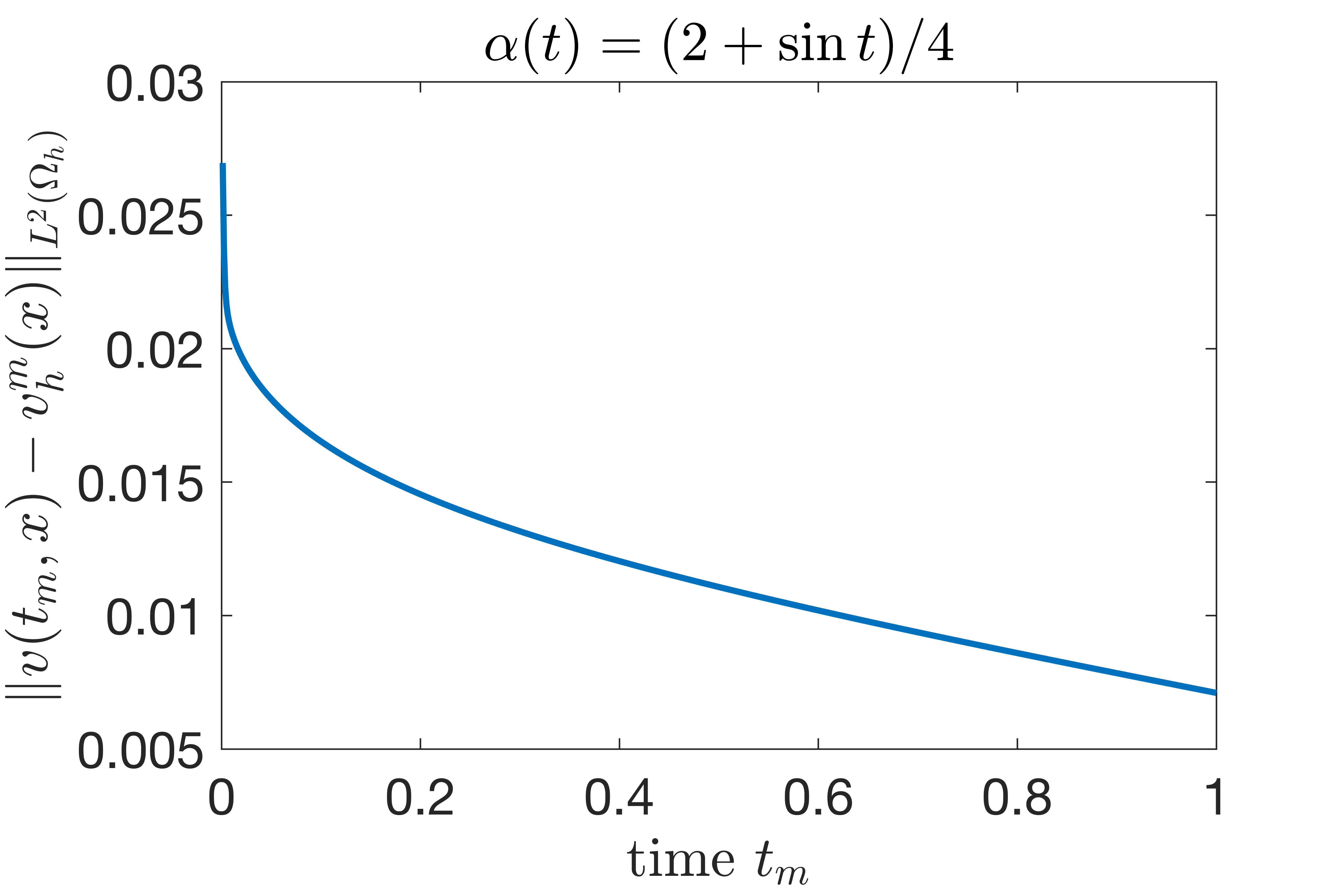} \quad 
\includegraphics[width=0.43\textwidth]{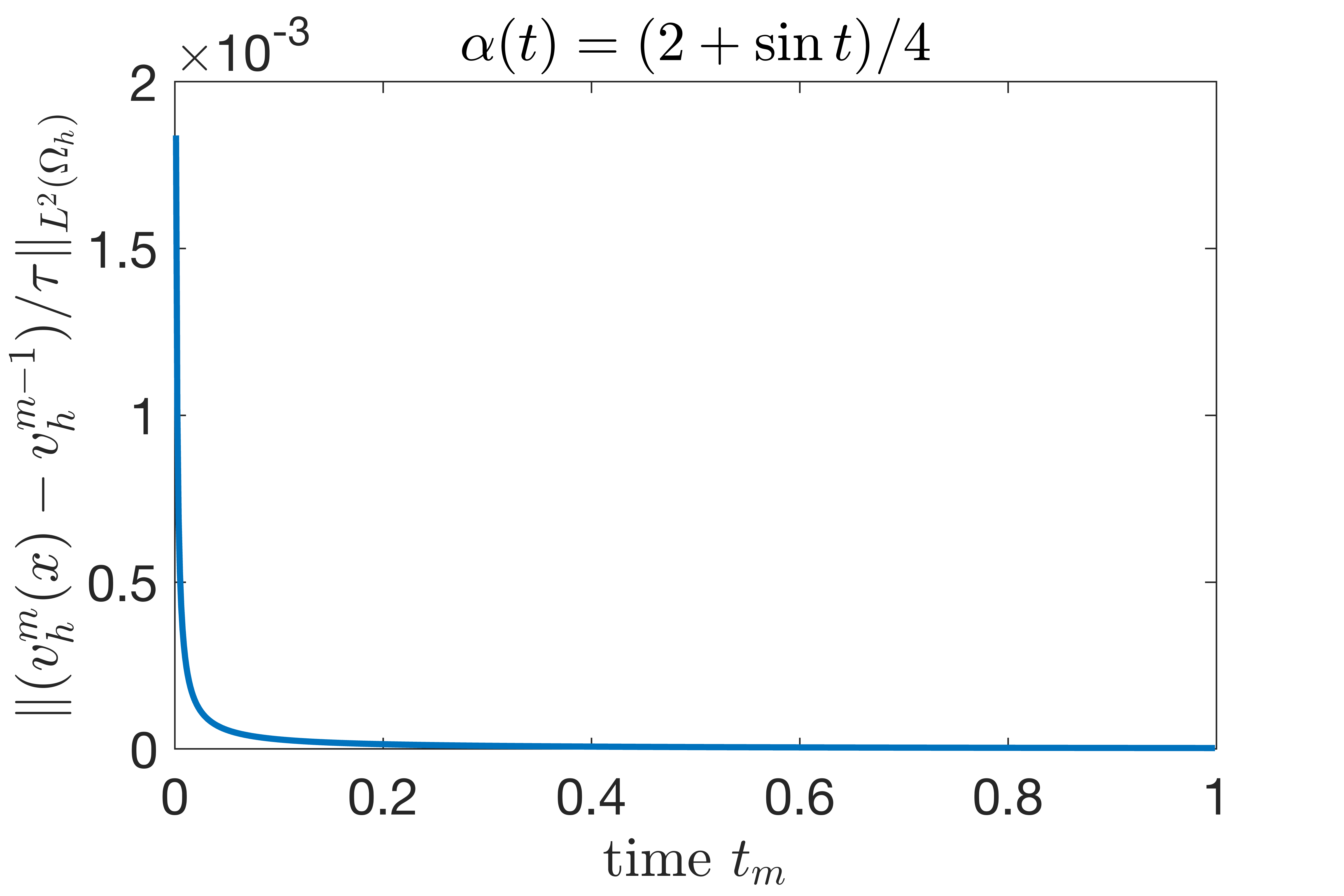}
\caption{\scriptsize{Time histories for the weakly singular solution \eqref{weak_singular} with \(M=800\) until \(T=1\); and \(\mathcal P^5\) polynomials with spatial mesh $N=200$. \textbf{Left:} \(L^2\)-error \(\|v(t_m)-v_h^m\|_{L^2(\Omega_h)}\). \textbf{Right:} \(L^2\)-norm of the backward difference \(\|(v_h^m-v_h^{m-1})/\tau\|_{L^2(\Omega_h)}\). The top and bottom rows correspond to \(\alpha(t)=0.1+0.8e^{-t}\) and \(\alpha(t)=(2+\sin t)/4\), respectively. The pronounced initial-time behavior is consistent with the weak singularity \(u_{tt}\sim t^{-1/2}\).}} \label{fig:time_history_weak_singularity}
\end{figure}

\subsection{Example in 2D}
In this section, we investigate the convergence of the proposed fully discrete energy-based DG scheme \eqref{eq:DG_fully1}--\eqref{eq:DG_fully1_m} for the variable-order time-fractional wave equation in two-dimensional space. Since the optimal convergence result in multiple dimensions is established for Cartesian grids, we report here numerical results obtained on Catersian grids. We consider the problem
 \begin{equation}\label{eq:prob2}
 \frac{\partial^2 u}{\partial t^2}+{}_0^C \mathcal{D}_t^{\alpha(t)+1} u =  \Delta u + f(x, y, t), 
\quad (x, y) \in (0, 1)\times(0, 1),\quad t\in (0, 1], 
 \end{equation}
 where $\alpha(t)$ satisfies the condition \eqref{eq:condition_alpha}. The exact solution is chosen in the form of
 \begin{equation}\label{eq:sol_form2}
 u(x,y,t) = G(t)\Phi_2(x,y),
 \end{equation}
 where \(G(t)\) is the same temporal profile as in the preceding one-dimensional test, and
{\[\Phi_2(x,y) = \Big(1 + \frac{1}{4}\cos 2\pi x + \frac{1}{5}\sin 2\pi y\Big)\sin 2\pi x \sin 2\pi y.\]}
 The source term \(f(x,y,t)\), together with the initial and boundary data, is determined by substituting \eqref{eq:sol_form2} into \eqref{eq:prob2}. Periodic boundary conditions are imposed. {In this example, we take  $\alpha(t) = 0.3+0.4|t-0.5|$. This variable order is Lipschitz continuous, but it is not $\mathcal{C}^1([0,1])$ because of the corner at $t = 0.5$.} We have also tested the other smooth variable-order functions used in the one-dimensional experiments and observed essentially the same convergence behavior; the corresponding tables are omitted to avoid repetition. 

The spatial discretization is performed on a uniform Cartesian grid with nodes \((x_k, y_j) = (kh, jh)\) for \(k, j = 0, 1, \cdots, N\) and \(h = 1/N\). The time interval is uniformly discretized with nodes \(t_m = m\tau\) for \(m = 0, \cdots, M\), where \(\tau = 1/M\). As in the one-dimensional experiments, we consider approximation spaces of degree \(q\in\{1,2,3,4\}\) for \(u_h\), $q-1$ for $v_h$, and the flux parameters $\theta = \gamma = \zeta =0$.   

Table \ref{tab:convergence_2d_space_u_al} presents the \(L^2\) errors for \(u\).  As in the one-dimensional spatial convergence test, we take \(\tau=10^{-4}\) so that the spatial discretization error dominates the temporal discretization error. The results show the optimal spatial convergence rate of order \(q+1\) for \(u\), which is consistent with the theoretical accuracy of the spatial DG discretization. To further validate the temporal accuracy of the proposed scheme, Table \ref{tab:convergence_2d_time_al} presents the \(L^2\) errors for both \(u\) and $v$ under different temporal mesh sizes. The spatial mesh size is fixed as \(h=1/20\) in each coordinate direction, and the polynomial degree is chosen as \(q=5\), so that the spatial discretization error is negligible compared with the temporal error. The number of time steps is varied as \(M=10,20,30,40\). Consistent with the one-dimensional results, the observed temporal convergence rates are $2$ for all cases.

\begin{table}[h!]
 	\footnotesize
 \begin{center}
 		\scalebox{1.0}{
 		\begin{tabular}{ c| c c c c }
 				\hline
     $q$ & $N\times N$ & $\vert$ & $E_u(h,\tau)$ & $\text{order}_h$  \\
 				\hline
 \cellcolor{gray!15} ~ & $10\times10$  & $\vert$ & 1.23e-01& \cellcolor{gray!15}--      \\
 \cellcolor{gray!15} ~ & $15\times15$  & $\vert$ & 5.56e-02& \cellcolor{gray!15}1.96  \\
 \cellcolor{gray!15}1  & $20\times20$  & $\vert$ & 3.15e-02& \cellcolor{gray!15}1.98\\
\cellcolor{gray!15} ~ & $25\times25$  & $\vert$ & 2.02e-02& \cellcolor{gray!15}1.99  \\
 			 ~ & ~  & ~ & ~ & ~ \\
 \cellcolor{cyan!5} ~  & $10\times10$   & $\vert$ & 9.77e-03& \cellcolor{cyan!5}--     \\
\cellcolor{cyan!5} ~  & $15\times15$  & $\vert$ & 2.95e-03& \cellcolor{cyan!5}2.95  \\
 \cellcolor{cyan!5}2   & $20\times20$  & $\vert$ & 1.25e-03& \cellcolor{cyan!5}2.98  \\
 \cellcolor{cyan!5} ~  & $25\times25$  & $\vert$ & 6.44e-04& \cellcolor{cyan!5}2.99 \\
 				\hline
 \end{tabular}
 		} \qquad\quad \scalebox{1.0}{
 		\begin{tabular}{ c| c c c c }
 				\hline
     $q$ & $N\times N$ & $\vert$ & $E_u(h,\tau)$ & $\text{order}_h$  \\
 				\hline
 \cellcolor{blue!5} ~ & $10\times10$  & $\vert$ & 3.51e-04& \cellcolor{gray!15}--      \\
 \cellcolor{blue!5} ~ & $15\times15$  & $\vert$ & 7.10e-05& \cellcolor{gray!15}3.94  \\
 \cellcolor{blue!5}3  & $20\times20$  & $\vert$ & 2.27e-05& \cellcolor{gray!15}3.97\\
\cellcolor{blue!5} ~ & $25\times25$  & $\vert$ & 9.35e-06& \cellcolor{gray!15}3.97  \\
 			 ~ & ~  & ~ & ~ & ~ \\
 \cellcolor{magenta!5} ~  & $10\times10$   & $\vert$ & 1.72e-05& \cellcolor{cyan!5}--     \\
\cellcolor{magenta!5} ~  & $15\times15$  & $\vert$ & 2.29e-05& \cellcolor{cyan!5}4.97  \\
 \cellcolor{magenta!5}4   & $20\times20$  & $\vert$ & 5.47e-06& \cellcolor{cyan!5} 4.98 \\
 \cellcolor{magenta!5} ~  & $25\times25$  & $\vert$ & 1.79e-06& \cellcolor{cyan!5} 5.00\\
 				\hline
 \end{tabular}
 		}
 \end{center}
\caption{\scriptsize{$L^2$ errors and the corresponding convergence rates for the problem \eqref{eq:prob2} of $u$ were computed using $\mathcal{P}^q$ polynomials with $q=1, 2, 3, 4$ and a time step size of $\tau = 10^{-4}$. The spatial domain was discretized into $N$ uniform cells with a terminal time of $T=1$. The numerical fluxes parameters are set to be $\theta = \gamma = \zeta = 0$, and the variable order $\alpha(t) = 0.3+0.4|t-1/2|$.}}\label{tab:convergence_2d_space_u_al}
 \end{table}

\begin{table}[h!]
 	\footnotesize
 \begin{center}
 		\scalebox{1.0}{
 		\begin{tabular}{ c c c c }
 				\hline
                 &  & $u$ &  \\
                \hline
 				$M$ & $\vert$ & $E_u(h,\tau)$ & $\text{order}_\tau$  \\
 				\cline{1-4}
 10 & $\vert$ &6.98e-03&   -- \\
 20 & $\vert$ & 1.78e-03& 1.97  \\
 30 & $\vert$ &7.89e-04& 2.01  \\
40 & $\vert$ &4.43e-04& 2.01 \\
 				\hline
 \end{tabular}
 		}\qquad \quad \scalebox{1.0}{
 		\begin{tabular}{ c c c c }
        \hline
                 &  &  $v$ &  \\
 				\hline
 				$M$ & $\vert$ & $E_v(h,\tau)$ & $\text{order}_\tau$  \\
 				\cline{1-4}
 10 & $\vert$ &4.67e-02&   -- \\
 20 & $\vert$ & 1.13e-02& 2.05  \\
 30 & $\vert$ &5.02e-03& 2.00  \\
40 & $\vert$ &2.83e-03& 1.99 \\
 				\hline
 \end{tabular}
 		}
 \end{center}
\caption{\scriptsize{$L^2$ errors and corresponding convergence rates for $(u,v)$ in the problem \eqref{eq:prob2} at time $T = 1$ were computed using $\mathcal{P}^5$ polynomials on a uniform Cartesian mesh of $20 \times 20$ elements with $\theta = \gamma = \zeta = 0$. The time interval was divided into $M$ uniform cells. The variable order is taken to be $\alpha(t) = 0.3+0.4|t-1/2|$.}}\label{tab:convergence_2d_time_al}
 \end{table}

\section{Brief Discussion}\label{sec:discussion}

In conclusion, we have developed and analyzed a fully discrete energy-based DG method for the Caputo-type variable-order time-fractional wave equation \eqref{eq:problem}.  Such wave equations can model complex wave phenomena by capturing memory effects and anomalous temporal dynamics more accurately in evolving systems, and their applications have been seen in seismology, biomechanics, materials science, electromagnetics, anomalous transport, and acoustics.  

To address the numerical challenges caused by variable-order time-fractional derivatives, we propose a fully discrete scheme that achieves second-order accuracy in time for sufficiently smooth solutions. For the spatial discretization, we introduce the first-order time derivative as an auxiliary variable and reformulate the original equation as a reduced-order system with only one additional unknown, independently of the spatial dimension. Compared with a local DG formulation in multidimensional space, this approach reduces both the memory requirements and the number of equations solved at each time step.

{A main analytical contribution of this work is the stability and convergence analysis for general variable fractional orders. In particular, by establishing a cumulative weight-variation estimate for the variable-order memory weights, we remove the monotonicity restriction on the fractional order that is often used to compare memory weights at consecutive time levels. Under the assumption that the order function $\alpha:[0,T] \rightarrow (0,1)$ is Lipschitz continuous, we establish energy stability of the fully discrete scheme and prove second-order temporal convergence.} We also derive spatial error estimates in an energy norm on structured grids. For Cartesian elements, with suitable numerical fluxes and appropriately paired approximation spaces for the displacement variable \(u\) and the velocity variable \(v\), the proposed scheme achieves optimal spatial convergence in the energy norm.

The numerical experiments support the theoretical findings. {In addition to standard convergence tests, we also investigate a weakly singular solution near the initial time. The results show the expected order reduction under reduced temporal regularity, illustrating the importance of initial-time singularities for Caputo-type fractional wave equations. A systematic analysis and implementation of nonuniform or graded temporal meshes for recovering higher-order convergence in the presence of such singularities will be pursued in future work.}

\section*{Appendix}
{We provide the detailed proof for Lemma \ref{lem:cumulative_weight_variation}. To this end, we first prove a Lipschitz estimate for the intermediate points. By Lemma \ref{lemma:nonlinear}, if $L_\alpha \tau < 2$, there exists a unique $\sigma_m$ such that
\[\sigma_m=1-\frac12\alpha(t_m+\sigma_m\tau),\]
where $L_\alpha$ is the Lipschitz constant of $\alpha(t)$. Hence
\[|\sigma_m-\sigma_{m-1}|
\le
\frac12 L_\alpha
\left|t_m+\sigma_m\tau-(t_{m-1}+\sigma_{m-1}\tau)\right|\le
\frac12L_\alpha
\left(\tau+\tau|\sigma_m-\sigma_{m-1}|\right).\]
Therefore, for sufficiently small $\tau$ (e.g., $L_\alpha \tau \leq 1$), we obtain
\begin{equation}\label{eq: aux5}
|\sigma_m-\sigma_{m-1}|
    \le C\tau \qquad \text{and}\qquad |\alpha_{m+\sigma_m}-\alpha_{m-1+\sigma_{m-1}}|
    = 2|\sigma_m -\sigma_{m-1}| 
    \le C\tau,
    \end{equation}
where $C$ is a positive constant independent of the time step size $\tau$ and the time level $m$. We next view the coefficients $a_i^{(m)}$ as functions of the two parameters $\alpha_{m+\sigma_m}$ and $\sigma_m$. For ease of notation, we write these parameters simply as $\alpha$ and $\sigma$.}

{
For $i=0$, define
\[
    a_0(\alpha,\sigma;\tau)
    :=
    \frac{\tau^{-\alpha}}{\Gamma(2-\alpha)}
    \left\{
    \frac{(1+\sigma)^{2-\alpha}-\sigma^{2-\alpha}}{2-\alpha}
    -\frac12\Big((1+\sigma)^{1-\alpha}-\sigma^{1-\alpha}\Big)
    \right\}.
\]
Since $\sigma\in(1/2,1)$ and $\alpha : [0,T]\rightarrow  (0,1)$ is Lipschitz continuous, the expression inside the braces and its first derivatives with respect to $\alpha$ and $\sigma$ are uniformly bounded. Moreover, differentiating $\tau^{-\alpha}$ with respect to $\alpha$ produces the factor $\log\tau$. Therefore, for $0<\tau\le 1$, there exists a positive constant $C$ independent of $\tau$ such that
\begin{equation}\label{eq: aux3}
    |\partial_\alpha a_0(\alpha,\sigma;\tau)|
    +|\partial_\sigma a_0(\alpha,\sigma;\tau)|
    \le
    C(1+|\log\tau|)\tau^{-\alpha_{\max}},
\end{equation}
where $\alpha_{\max} := \max_{t\in[0,T]} \alpha(t) \in (0,1)$.}

{For $1\leq i\leq m-1$, denote $\Delta^2 f(x):=f(x+1)-2f(x)+f(x-1)$. Then by \eqref{eq:cdef}, we have
\[
    c_i(\alpha,\sigma)
    =
    \frac{\Delta^2 \big((i+\sigma)^{2-\alpha}\big)}{2-\alpha}
    -
    \frac12\Delta^2 \big((i+\sigma)^{1-\alpha}\big).
\]
Using the identity $\Delta^2 f(x)
    =
    \int_{-1}^{1}(1-|y|)f''(x+y)\,dy$ yields
\begin{equation}\label{eq:auxx1}
    a_i(\alpha,\sigma;\tau)
    =
    \frac{1-\alpha}{\Gamma(2-\alpha)}
    \int_{-1}^{1}
    (1-|y|)
    \bigl((i+\sigma+y)\tau\bigr)^{-\alpha}
    +
    \frac{\alpha \tau}{2}(1-|y|)
    \bigl((i+\sigma+y)\tau\bigr)^{-1-\alpha}\,dy.
\end{equation}
Since $1\leq i\leq m-1$ and $\sigma\in(1/2,1)$, we have
\begin{equation*}
    \frac14(i+1)\le i+\sigma+y\le \frac32(i+1),
    \qquad -1\le y\le 1.
\end{equation*}
Thus, there exists a positive constant $C > 0$ independent of $\tau$ and $m$ such that
\begin{equation}\label{eq: aux4}
\left|\log\bigl((i+\sigma+y)\tau\bigr)\right|
    \le C\left(1+\left|\log((i+1)\tau)\right|\right).
\end{equation}
Now, differentiating \eqref{eq:auxx1} with respect to $\alpha$ and $\sigma$, and using \eqref{eq: aux4}, yields
\begin{equation}\label{eq:auxx2}
    |\partial_\alpha a_i(\alpha,\sigma;\tau)|
    +
    |\partial_\sigma a_i(\alpha,\sigma;\tau)|
    \le
    C
    \left(1+\left|\log((i+1)\tau)\right|\right)
    ((i+1)\tau)^{-\alpha_{\max}}.
\end{equation}
Combining \eqref{eq:auxx2} with \eqref{eq: aux5}--\eqref{eq: aux3}, and using the mean value theorem, yields \eqref{eq: aux1}. It remains to prove the cumulative estimate \eqref{eq: aux2}. Since
\[
    \big(a_{m-k}^{(m)}-a_{m-k}^{(m-1)}\big)_+ = \max\{a_{m-k}^{(m)}-a_{m-k}^{(m-1)}, 0\} 
    \le \big|a_{m-k}^{(m)}-a_{m-k}^{(m-1)}\big|,
\]
we obtain from \eqref{eq: aux1} that for any $2\leq k\leq n\leq M-1$
\[\begin{aligned}
    \tau\sum_{m=k}^{n}
    \left(a_{m-k}^{(m)}-a_{m-k}^{(m-1)}\right)_+ & \le
C\tau^2
\sum_{m=k}^{n}
\left(1+\left|\log((m-k+1)\tau)\right|\right)
\big((m-k+1)\tau\big)^{-\alpha_{\max}} \\
   & =
    C\tau^2
\sum_{i=0}^{n-k}
\left(1+\left|\log((i+1)\tau)\right|\right)
\big((i+1)\tau\big)^{-\alpha_{\max}}.
    \end{aligned}
\]
Let \(g(r):=(1+|\log r|)r^{-\alpha_{\max}}\). Since
\(\alpha_{\max}<1\), \(g\in L^1(0,T)\). Moreover, \(g\) is decreasing
near \(r=0\) and bounded away from \(r=0\). Therefore, provided
\((n-k+1)\tau\le T\), there exists a positive constant \(C>0\), independent of \(\tau,n,k\), such that
\[
    \tau
    \sum_{i=0}^{n-k} \left(1+\left|\log((i+1)\tau)\right|\right)
\big((i+1)\tau\big)^{-\alpha_{\max}}
    \le C.
\]
Therefore, we obtain the cumulative estimate \eqref{eq: aux2}.
}

\section*{Acknowledgement}
This work was partially supported by the NSF grants \#DMS-2231482 and \#DMS-2513924, Simons Foundation through travel support for mathematicians, and Ken Kennedy Institute at Rice University.

\bibliography{lu,BIB-Local}

@article{zhang2021local,
  title={A local energy-based discontinuous Galerkin method for fourth-order semilinear wave equations},
  author={Zhang, Lu},
  journal={IMA J. Numer. Anal.},
  volume={44},
  number={5},
  pages={2793-2820},
  year={2024},
  publisher={Oxford University Press}
}

@article{chou2014optimal,
	title={Optimal energy conserving local discontinuous {G}alerkin methods for second-order wave equation in heterogeneous media},
	author={Chou, Ching-Shan and Shu, Chi-Wang and Xing, Yulong},
	journal={J. Comput. Phys.},
	volume={272},
	pages={88-107},
	year={2014},
	publisher={Elsevier}
}

@article{riviere2003discontinuous,
	title={Discontinuous finite element methods for acoustic and elastic wave problems},
	author={Riviere, Beatrice and Wheeler, Mary F},
	journal={Contemp. Math.},
	volume={329},
	number={271-282},
	pages={4-6},
	year={2003},
	publisher={Providence, RI: American Mathematical Society}
}

@article{zhang2019energy,
	title={An energy-based discontinuous {G}alerkin method for the wave equation with advection},
	author={Zhang, Lu and Hagstrom, Thomas and Appel{\"o}, Daniel},
	journal={SIAM J. Numer. Anal.},
	volume={57},
	number={5},
	pages={2469-2492},
	year={2019},
	publisher={SIAM}
}

@article{appelo2018energy,
	title={An energy-based discontinuous {G}alerkin discretization of the elastic wave equation in second order form},
	author={Appel{\"o}, Daniel and Hagstrom, Thomas},
	journal={Comput. Methods Appl. Mech. Engrg.},
	volume={338},
	pages={362-391},
	year={2018},
	publisher={Elsevier}
}

@article{appelo2015new,
	title={A new discontinuous {G}alerkin formulation for wave equations in second-order form},
	author={Appel{\"o}, Daniel and Hagstrom, Thomas},
	journal={SIAM J. Numer. Anal.},
	volume={53},
	number={6},
	pages={2705-2726},
	year={2015},
	publisher={SIAM}
}

@article{appelo2020energy,
	title={An energy-based discontinuous {G}alerkin method for semilinear wave equations},
	author={Appel{\"o}, Daniel and Hagstrom, Thomas and Wang, Qi and Zhang, Lu},
	journal={J. Comput. Phys.},
	volume={418},
	pages={109608},
	year={2020},
	publisher={Elsevier}
}

@article{du2019convergence,
	title={Convergence analysis of a discontinuous {G}alerkin method for wave equations in second-order form},
	author={Du, Yu and Zhang, Lu and Zhang, Zhimin},
	journal={SIAM J. Numer. Anal.},
	volume={57},
	number={1},
	pages={238-265},
	year={2019},
	publisher={SIAM}
}

@article{zhang2021energy,
  title={Energy-based discontinuous {G}alerkin difference methods for second-order wave equations},
  author={Zhang, Lu and Appel{\"o}, Daniel and Hagstrom, Thomas},
  journal={Commun. Appl. Math. Comput.},
  pages={1-25},
  year={2021},
  publisher={Springer}
}

@book{ciarlet2002finite,
	title={The Finite Element Method for Elliptic Problems},
	author={Ciarlet, Philippe G},
	year={2002},
	publisher={SIAM}
}

@article{appelo2021stagger, 
  title={An energy-based discontinuous {Galerkin} method with tame {CFL} numbers for the wave equation},
  author={Appel{\"o}, Daniel and Zhang, Lu and Hagstrom, Thomas and Li, Fengyan},
  journal={BIT Numer. Math},
  volume={63},
  number={1},
  pages={5},
  year={2023},
}

@inproceedings{hagstrom2021discontinuous,
  title={Discontinuous {G}alerkin Methods for Electromagnetic Waves in Dispersive Media},
  author={Hagstrom, Thomas and Appel{\"o}, Daniel and Zhang, Lu},
  booktitle={2021 International Applied Computational Electromagnetics Society Symposium (ACES)},
  pages={1-4},
  year={2021},
  organization={IEEE}
}

@article{heydari2019wavelet,
  title={A wavelet approach for solving multi-term variable-order time fractional diffusion-wave equation},
  author={Heydari, Mohammad Hossein and Avazzadeh, Zakieh and Haromi, Malih Farzi},
  journal={Appl. Math. Comput.},
  volume={341},
  pages={215-228},
  year={2019},
  publisher={Elsevier}
}

@article{zeng2017generalized,
  title={A generalized spectral collocation method with tunable accuracy for fractional differential equations with end-point singularities},
  author={Zeng, Fanhai and Mao, Zhiping and Karniadakis, George Em},
  journal={SIAM J. Sci. Comput.},
  volume={39},
  number={1},
  pages={A360-A383},
  year={2017},
  publisher={SIAM}
}

@article{shekari2019meshfree,
  title={A meshfree approach for solving 2{D} variable-order fractional nonlinear diffusion-wave equation},
  author={Shekari, Younes and Tayebi, Ali and Heydari, Mohammad Hossein},
  journal={Comput. Methods Appl. Mech. Eng},
  volume={350},
  pages={154-168},
  year={2019},
  publisher={Elsevier}
}

@article{tayebi2017meshless,
  title={A meshless method for solving two-dimensional variable-order time fractional advection-diffusion equation},
  author={Tayebi, Ali and Shekari, Younes and Heydari, Mohammad Hossein},
  journal={J. Comput. Phys.},
  volume={340},
  pages={655-669},
  year={2017},
  publisher={Elsevier}
}

@article{haq2019numerical,
  title={Numerical solutions of variable order time fractional $(1+ 1)$-and $(1+ 2)$-dimensional advection dispersion and diffusion models},
  author={Haq, Sirajul and Ghafoor, Abdul and Hussain, Manzoor},
  journal={Appl. Math. Comput.},
  volume={360},
  pages={107-121},
  year={2019},
  publisher={Elsevier}
}

@article{fu2019robust,
  title={A robust kernel-based solver for variable-order time fractional {PDE}s under 2{D}/3{D} irregular domains},
  author={Fu, Zhuo-Jia and Reutskiy, Sergiy and Sun, Hong-Guang and Ma, Ji and Khan, Mushtaq Ahmad},
  journal={Appl. Math. Lett.},
  volume={94},
  pages={105-111},
  year={2019},
  publisher={Elsevier}
}

@article{heydari2019computational,
  title={A computational method for solving variable-order fractional nonlinear diffusion-wave equation},
  author={Heydari, Mohammad Hossein and Avazzadeh, Zakieh and Yang, Yin},
  journal={Appl. Math. Comput.},
  volume={352},
  pages={235-248},
  year={2019},
  publisher={Elsevier}
}

@article{garrappa2023computational,
  title={A computational approach to exponential-type variable-order fractional differential equations},
  author={Garrappa, Roberto and Giusti, Andrea},
  journal={J. Sci. Comput.},
  volume={96},
  number={3},
  pages={63},
  year={2023},
  publisher={Springer}
}

@article{du2020temporal,
  title={Temporal second order difference schemes for the multi-dimensional variable-order time fractional sub-diffusion equations},
  author={Du, Ruilian and Alikhanov, Anatoly A and Sun, Zhizhong},
  journal={Comput. Math. Appl.},
  volume={79},
  number={10},
  pages={2952-2972},
  year={2020},
  publisher={Elsevier}
}

@article{alikhanov2015new,
  title={A new difference scheme for the time fractional diffusion equation},
  author={Alikhanov, Anatoly A},
  journal={J. Comput. Phys.},
  volume={280},
  pages={424-438},
  year={2015},
  publisher={Elsevier}
}

@article{du2022temporal,
  title={Temporal second-order finite difference schemes for variable-order time-fractional wave equations},
  author={Du, Ruilian and Sun, Zhizhong and Wang, Hong},
  journal={SIAM J. Numer. Anal.},
  volume={60},
  number={1},
  pages={104-132},
  year={2022},
  publisher={SIAM}
}

@article{sun2016some,
  title={Some temporal second order difference schemes for fractional wave equations},
  author={Sun, Hong and Sun, Zhizhong and Gao, Guanghua},
  journal={Numer. Methods Partial Differ. Equ.},
  volume={32},
  number={3},
  pages={970-1001},
  year={2016},
  publisher={Wiley Online Library}
}

@article{bagley1983theoretical,
  title={A theoretical basis for the application of fractional calculus to viscoelasticity},
  author={Bagley, Ronald L and Torvik, Peter J},
  journal={J. Rheol.},
  volume={27},
  number={3},
  pages={201-210},
  year={1983},
  publisher={The Society of Rheology}
}

@article{metzler2000random,
  title={The random walk's guide to anomalous diffusion: a fractional dynamics approach},
  author={Metzler, Ralf and Klafter, Joseph},
  journal={Phys. Rep.},
  volume={339},
  number={1},
  pages={1-77},
  year={2000},
  publisher={Elsevier}
}

@book{pipkin2012lectures,
  title={Lectures on Viscoelasticity Theory},
  author={Pipkin, Allen C},
  volume={7},
  year={2012},
  publisher={Springer Science \& Business Media}
}

@article{zaslavsky2002chaos,
  title={Chaos, fractional kinetics, and anomalous transport},
  author={Zaslavsky, George M},
  journal={Phys. Rep.},
  volume={371},
  number={6},
  pages={461-580},
  year={2002},
  publisher={Elsevier}
}

@article{zhuang2009numerical,
  title={Numerical methods for the variable-order fractional advection-diffusion equation with a nonlinear source term},
  author={Zhuang, Pinghui and Liu, Fawang and Anh, Vo and Turner, Ian},
  journal={SIAM J. Numer. Anal.},
  volume={47},
  number={3},
  pages={1760-1781},
  year={2009},
  publisher={SIAM}
}

@article{coimbra2003mechanics,
  title={Mechanics with variable-order differential operators},
  author={Coimbra, Carlos FM},
  journal={Ann. Phys. (Berlin)},
  volume={515},
  number={11-12},
  pages={692-703},
  year={2003},
  publisher={Wiley Online Library}
}

@book{podlubny1998fractional,
  title={Fractional Differential Equations: An Introduction to Fractional Derivatives, Fractional Differential Equations, to Methods of Their Solution and Some of Their Applications},
  author={Podlubny, Igor},
  year={1998},
  publisher={Elsevier}
}

@article{nigmatullin1986realization,
  title={The realization of the generalized transfer equation in a medium with fractal geometry},
  author={Nigmatullin, RR},
  journal={Phys. Status Solidi (b)},
  volume={133},
  number={1},
  pages={425-430},
  year={1986},
  publisher={Wiley Online Library}
}

@article{lin2009stability,
  title={Stability and convergence of a new explicit finite-difference approximation for the variable-order nonlinear fractional diffusion equation},
  author={Lin, Ran and Liu, Fawang and Anh, Vo and Turner, Ian},
  journal={Appl. Math. Comput.},
  volume={212},
  number={2},
  pages={435-445},
  year={2009},
  publisher={Elsevier}
}

@article{zheng2021optimal,
  title={Optimal-order error estimates of finite element approximations to variable-order time-fractional diffusion equations without regularity assumptions of the true solutions},
  author={Zheng, Xiangcheng and Wang, Hong},
  journal={IMA J. Numer. Anal.},
  volume={41},
  number={2},
  pages={1522-1545},
  year={2021},
  publisher={Oxford University Press}
}

@article{lei2023finite,
  title={Finite element discretizations for variable-order fractional diffusion problems},
  author={Lei, Wenyu and Turkiyyah, George and Knio, Omar},
  journal={J. Sci. Comput.},
  volume={97},
  number={1},
  pages={5},
  year={2023},
  publisher={Springer}
}

@article{zeng2015generalized,
  title={A generalized spectral collocation method with tunable accuracy for variable-order fractional differential equations},
  author={Zeng, Fanhai and Zhang, Zhongqiang and Karniadakis, George Em},
  journal={SIAM J. Sci. Comput.},
  volume={37},
  number={6},
  pages={A2710-A2732},
  year={2015},
  publisher={SIAM}
}

@article{zhao2019multi,
  title={Multi-domain spectral collocation method for variable-order nonlinear fractional differential equations},
  author={Zhao, Tinggang and Mao, Zhiping and Karniadakis, George Em},
  journal={Comput. Methods Appl. Mech. Eng.},
  volume={348},
  pages={377-395},
  year={2019},
  publisher={Elsevier}
}

@article{xu2023novel,
  title={A novel meshless method based on {R}{B}{F} for solving variable-order time fractional advection-diffusion-reaction equation in linear or nonlinear systems},
  author={Xu, Yi and Sun, HongGuang and Zhang, Yuhui and Sun, Hai-Wei and Lin, Ji},
  journal={Comput. Math. Appl.},
  volume={142},
  pages={107-120},
  year={2023},
  publisher={Elsevier}
}

@article{bhrawy2016space,
  title={A space-time spectral collocation algorithm for the variable order fractional wave equation},
  author={Bhrawy, AH and Doha, EH and Alzaidy, JF and Abdelkawy, MA},
  journal={Springer Plus},
  volume={5},
  pages={1-15},
  year={2016},
  publisher={Springer}
}

@article{zheng2022analysis,
  title={Analysis and discretization of a variable-order fractional wave equation},
  author={Zheng, Xiangcheng and Wang, Hong},
  journal={Commun. Nonlinear Sci. Numer. Simul.},
  volume={104},
  pages={106047},
  year={2022},
  publisher={Elsevier}
}

@article{ren2023energy,
  title={An energy-based discontinuous {G}alerkin method for the nonlinear {S}chr{\"o}dinger equation with wave operator},
  author={Ren, Kui and Zhang, Lu and Zhou, Yin},
  journal={SIAM J. Numer. Anal.},
  volume={62},
  number={6},
  pages={2459-2483},
  year={2024},
  publisher={SIAM}
}

@book{ern2021finiteI,
  author    = {Ern, Alexandre and Guermond, Jean-Luc},
  title     = {Finite Elements I: Approximation and Interpolation},
  series    = {Texts in Applied Mathematics},
  volume    = {72},
  publisher = {Springer},
  address   = {Cham},
  year      = {2021}
}

\bibliographystyle{abbrv}

\end{document}